\documentclass{amsart}
\usepackage{amsmath, amsthm, amssymb}
\usepackage{hyperref} 
\usepackage{enumerate}
\usepackage{verbatim}
\usepackage{esint}
\usepackage{color}
\usepackage[T1]{fontenc}

\usepackage{tikz,amsthm,amsmath,amstext,amssymb,amscd,epsfig,euscript, mathrsfs,
multicol,graphpap,graphics,graphicx,times,enumerate,subfig,
wrapfig,color,pict2e}
\usepackage{setspace}

\numberwithin{equation}{section}

\title{A sharp necessary condition for rectifiable curves in metric spaces} 
\author{Guy C. David}
\address{Department of Mathematical Sciences\\ Ball State University, Muncie, IN 47306}
\email{gcdavid@bsu.edu}

\author{Raanan Schul}
\address{Department of Mathematics\\ Stony Brook University\\ Stony Brook, NY 11794-3651}
\email{schul@math.sunysb.edu}

\date{\today}

\subjclass[2010]{28A75.}

\begin{document}
\maketitle

\theoremstyle{plain}
\newtheorem{theorem}{Theorem}
\newtheorem{exercise}{Exercise}
\newtheorem{corollary}[theorem]{Corollary}
\newtheorem{scholium}[theorem]{Scholium}
\newtheorem{claim}[theorem]{Claim}
\newtheorem{lemma}[theorem]{Lemma}
\newtheorem{sublemma}[theorem]{Lemma}
\newtheorem{proposition}[theorem]{Proposition}
\newtheorem{conjecture}{theorem}
\newtheorem{maintheorem}{Theorem}
\newtheorem{maincor}[maintheorem]{Corollary}
\renewcommand{\themaintheorem}{\Alph{maintheorem}}

\theoremstyle{definition}
\newtheorem{fact}[theorem]{Fact}
\newtheorem{example}[theorem]{Example}
\newtheorem{definition}[theorem]{Definition}
\newtheorem{remark}[theorem]{Remark}

\numberwithin{equation}{section}
\numberwithin{theorem}{section}

\newcommand{\cG}{\mathcal{G}}
\newcommand{\RR}{\mathbb{R}}
\newcommand{\HH}{\mathcal{H}}
\newcommand{\LIP}{\textnormal{LIP}}
\newcommand{\Lip}{\textnormal{Lip}}
\newcommand{\Tan}{\textnormal{Tan}}
\newcommand{\length}{\textnormal{length}}
\newcommand{\dist}{\textnormal{dist}}
\newcommand{\diam}{\textnormal{diam}}
\newcommand{\vol}{\textnormal{vol}}
\newcommand{\rad}{\textnormal{rad}}

\def\bA{{\mathbb{A}}}
\def\bB{{\mathbb{B}}}
\def\bC{{\mathbb{C}}}
\def\bD{{\mathbb{D}}}
\def\bR{{\mathbb{R}}}
\def\bS{{\mathbb{S}}}
\def\bO{{\mathbb{O}}}
\def\bE{{\mathbb{E}}}
\def\bF{{\mathbb{F}}}
\def\bH{{\mathbb{H}}}
\def\bI{{\mathbb{I}}}
\def\bT{{\mathbb{T}}}
\def\bZ{{\mathbb{Z}}}
\def\bX{{\mathbb{X}}}
\def\bP{{\mathbb{P}}}
\def\bN{{\mathbb{N}}}
\def\bQ{{\mathbb{Q}}}
\def\bK{{\mathbb{K}}}
\def\bG{{\mathbb{G}}}

\def\nrj{{\mathcal{E}}}
\def\cA{{\mathscr{A}}}
\def\cB{{\mathscr{B}}}
\def\cC{{\mathscr{C}}}
\def\cD{{\mathscr{D}}}
\def\cE{{\mathscr{E}}}
\def\cF{{\mathscr{F}}}
\def\cB{{\mathscr{G}}}
\def\cH{{\mathscr{H}}}
\def\cI{{\mathscr{I}}}
\def\cJ{{\mathscr{J}}}
\def\cK{{\mathscr{K}}}
\def\Layer{{\rm Layer}}
\def\cM{{\mathscr{M}}}
\def\cN{{\mathscr{N}}}
\def\cO{{\mathscr{O}}}
\def\cP{{\mathscr{P}}}
\def\cQ{{\mathscr{Q}}}
\def\cR{{\mathscr{R}}}
\def\cS{{\mathscr{S}}}
\def\Up{{\rm Up}}
\def\cU{{\mathscr{U}}}
\def\cV{{\mathscr{V}}}
\def\cW{{\mathscr{W}}}
\def\cX{{\mathscr{X}}}
\def\cY{{\mathscr{Y}}}
\def\cZ{{\mathscr{Z}}}

\def\tQ{{\tilde{Q}}}
	\def\ctQ{{\tilde{\cQ}}}
  \def\bt{{\tilde{\beta}}}
  \def\btc{{\tilde{\beta}_{c}}}
	\def\btcfm{{\tilde{\beta}_{\mathcal{F}_M}}}
	\def\btcf{{\tilde{\beta}_{\mathcal{F}}}}
	\def\bh{{\hat{\beta}_c}}
  \def\bc{{\beta_{c}}}
	\def\bcv{{\beta_{c,v}}}
	\def\bi{{\beta_{\infty}}}
	\def\bgh{{\beta_{GH}}}
	\def\bti{{\tilde{\beta}_{\infty}}}
	\def\bx{{\beta_{X}}}
  \def\del{\partial}
  \def\diam{{\rm diam}}
	\def\HHH{{\HH^1_\infty}}
	\def\HHHt{{(\tilde{\HH}^1_\infty)^3}}
	\def\HHN{{(\HH^1_\infty)^N}}
	\def\HT{{(\HH^1_\infty)^2}}
	\def\VV{{\mathcal{V}}}
	\def\FF{{\mathcal{F}}}
	\def\QQ{{\mathcal{Q}}}
	\def\BB{{\mathcal{B}}}
	\def\XX{{\mathcal{X}}}

  \def\eb{{\epsilon_\beta}}
  \def\del{\partial}
  \def\diam{{\rm diam}}
  \def\children{{\rm \mathcal{C}}}
	\def\image{{\rm Image}}
	\def\domain{{\rm Domain}}
  \def\dist{{\rm dist}}

\newcommand{\RS}[1]{{  \color{red} \textbf{Raanan:} #1}}
\newcommand{\GCD}[1]{{  \color{blue} \textbf{Guy:} #1}}

\begin{abstract}
In his 1990 Inventiones paper, P. Jones characterized subsets of rectifiable curves in the plane, using a multiscale sum of what is now known as Jones $\beta$-numbers, numbers measuring flatness in a given scale and location.  This work was generalized to $\bR^n$ by Okikiolu, to Hilbert space by the second author, and has many variants in a variety of metric settings.  Notably, in 2005, Hahlomaa gave a sufficient condition for a subset of a metric space to be contained in a rectifiable curve.   We prove the sharpest possible converse to Hahlomaa's theorem for doubling curves, and then deduce some corollaries for subsets of metric and Banach spaces, as well as the Heisenberg group.
\end{abstract}

\tableofcontents

\section{Introduction}

\subsection{Background}

In an Inventiones paper \cite{Jones-TSP}, Peter Jones proved the following theorem for sets in $\bR^2$, later generalized by Kate Okikiolu  to sets in $\bR^n$ \cite{Ok92}.
For  sets $E,B\subseteq \bR^{n}$,  define
\begin{equation}
\beta_{E,\infty}^{\rm Euc}(B)=\frac2{\diam(B)}\inf\limits_L\sup\{\dist(y,L):y\in E\cap B\}
\label{e:euclidean-beta}
\end{equation}
where $L$ ranges over lines  in $\bR^{n}$.

\begin{theorem}(Jones: $\bR^2$  \cite{Jones-TSP}; Okikiolu: $\bR^n$  \cite{Ok92}) Let  $n\geq 2$.
There is a $C=C(n)$ such that the following holds.
Let $E\subset \bR^n$.
Then there is  a connected set $\Gamma\supseteq E$ such that
\begin{equation}
\cH^{1}(\Gamma)\lesssim_{n} \diam(E)+\sum_{Q\in \Delta\atop Q\cap E\neq\emptyset} \beta_{E,\infty}^{\rm Euc}(3Q)^{2}\diam(Q).
\label{e:betaE}
\end{equation}
Conversely, if $\Gamma$ is connected and $\cH^{1}(\Gamma)<\infty$, then 
\begin{equation}
 \diam(\Gamma)+\sum_{Q\in \Delta\atop Q\cap \Gamma\neq\emptyset} \beta_{\Gamma,\infty}^{\rm Euc}(3Q)^{2}\diam(Q) \lesssim_{n}  \cH^{1}(\Gamma).
\label{e:beta_gamma}
\end{equation}
\label{t:TST}
\end{theorem}
Here, $\cH^k$ is the $k-$dimensional Hausdorff measure, and $\Delta$ is the collection of dyadic cubes in $\bR^n$.
Given two functions $a$ and $b$ into $\bR$ we say 
$a\lesssim b$ with constant $C$, when there exists a constant $C$ such that $a\leq C b$; the subscript $\lesssim_{n}$ indicates the dependence of the implied constant on $n$. We say that $a\sim b$ if  $a\lesssim b$ and $b\lesssim a$.

Equation \eqref{e:beta_gamma}, whose variants are the main subject of this paper, gives a ``quantitative flatness'' statement for rectifiable curves in $\bR^n$. In other words, a rectifiable curve in $\bR^n$ must lie close to a line at most locations and scales, in a very precise manner. This may be viewed as a quantitative version of the qualitative statement that a rectifiable curve has linear tangents almost everywhere along its length. In addition to the clear geometric information it provides, \eqref{e:beta_gamma} and its variants have had an important influence on the study of singular integrals from the 1980's onward. As a small sample of these connections, we point the reader to the works of Jones \cite{Jones87}, David-Semmes \cite{DS-sing-int-book, DS-blue-book} and Tolsa \cite{Tolsa-book}, as well as the survey \cite{Pajot-book}. There are also connections to recent developments in the study of harmonic measure (see, e.g., \cite{seven-author} and the many advances by the same authors).

Theorem \ref{t:TST} was later generalized to sets $E$ lying in Hilbert space \cite{Sc07_hilbert}, which requires replacing $\Delta$ by a multiscale collection of balls centered on the set in question. 
There are also variants in some metric settings \cite{Hahlomaa-non-AR, Hahlomaa-AR, Sc07_metric}, including the Heisenberg group \cite{LS14, LS15} and a collection of non-Euclidean metric spaces generalizing a construction of Laakso \cite{GCD-RS-TSP-Laakso}.  
We will return to some of these variants in more detail in the sections below, but for now we note that these variants require 
changes in the definition of the $\beta$-number (as there are no longer Euclidean lines), as well as for modifying the exponent $2$ in the sums analogous to those in \eqref{e:betaE} and \eqref{e:beta_gamma}.

\begin{remark}
A point that the authors find intriguing is that, for these variants of Theorem \ref{t:TST}, even when analogs of both \eqref{e:betaE} and \eqref{e:beta_gamma} are known, the exponents of $\beta$ in these results are generally not known to match. Hence, one does not always achieve a characterization of subsets of rectifiable curves by this method.
To our knowledge, the only settings where these exponents are known to match are in Hilbert spaces \cite{Sc07_hilbert}, in  Ahlfors $1$-regular metric spaces \cite{Hahlomaa-AR, Sc07_metric} (where an average replaces the supremum in the definition of $\beta$), and, to some extent, in the graph inverse limit spaces studied in \cite{GCD-RS-TSP-Laakso}. In other settings, such as the Heisenberg group and general Banach spaces, the situation is not completely clear.\footnote{A very recent preprint by Sean Li \cite{Li19}, posted to the arXiv on the same day as this paper, gives a version of Theorem \ref{t:TST} with matching exponents in general Carnot groups, including the Heisenberg group.}

\end{remark}

We  now wish to state a metric analogue of the first half of Theorem \ref{t:TST}. 
Given a metric space  $E$ and a ball $B=B(p,r)$, let
\begin{equation}\label{e:beta-metric}
\beta_\infty^E(B)^2  = r^{-1}\sup\{ \del(x,y,z) : x,y,z\in E \cap B\}. 
\end{equation}
(If the metric space is understood, we will drop the superscript.)

The quantity $\del(x,y,z)$ is is the defect in the triangle inequality, and is defined in Section \ref{sec:beta};  for now let us say that for a triple $x,y,z$ such that $\dist(x,y)\leq \dist(y,z)\leq \dist(x,z)$, the quantity 
$\del(x,y,z)$ is given by
  $\dist(x,y)+\dist(y,z)-\dist(x,z)$.
See 
Section \ref{sec:beta} for a more detailed definition of $\del$.

In particular, $\beta_\infty^E(B)$ gives a measurement of ``flatness'' that can be studied in general metric spaces without notions of lines. Thus, one can use this notion of $\beta$-number to study the validity of metric space analogs of Theorem \ref{t:TST}. In that case, one must replace the family of dyadic cubes in $\bR^n$ by a ``multiresolution family of balls'' $\cB$, as defined in in \ref{d:multiresolution} below.

Given these modifications, Hahlomaa proved the following analogue of \eqref{e:betaE}.

\begin{theorem}[\cite{Hahlomaa-non-AR}, Theorem 5.3]\label{t:Hah}
Let $E$ be a bounded metric space and $\cB$ a multiresolution family such that
\begin{equation}\label{eq:hahlomaa}
\sum_{B\in \cB} \bi(B)^2\diam(B) < \infty
\end{equation}
Then there is a set $A\subseteq [0,1]$ and a surjective Lipschitz map $f:A\rightarrow E$ with Lipschitz constant bounded by
\begin{equation}\label{eq:hahlomaa2}
C\left(\diam(E) + \sum_{B\in \cB} \bi(B)^2\diam(B)\right).
\end{equation}
The constant $C>0$ is absolute.
\end{theorem}
The notion of a multiresolution family $\cB$ is defined in Section \ref{d:multiresolution}. 

In fact, as stated, Theorem 5.3 of \cite{Hahlomaa-non-AR} uses a different $\beta$-number, defined using Menger curvature, rather than our $\bi$. However, Hahlomaa's definition of $\beta$ is bounded above by an absolute constant times that in \eqref{e:beta-metric}. See Remark \ref{r:sep-trip} below. Hence, the theorem above follows immediately from his work.
 We note that Menger curvature was further used as a sufficient condition for 1-rectifiability by L\'eger \cite{Leger} in the Euclidean setting and Hahlomma \cite{Hahlomaa-does-leger} in the metric setting. The book \cite{Pajot-book} gives a nice survey of this up to the time it was written. 
 
\begin{remark}\label{rmk:counterex}
The converse to Theorem \ref{t:Hah} is false, as is demonstrated in \cite[Example 3.3.1]{Sc07_survey}. In that example, a sequence of rectifiable curves $\{\Gamma_n\}$ is constructed in the Banach space $(\RR^2, \|\cdot\|_{\ell^1})$ so that $\HH^1(\Gamma_n)=2$ for all $n$ but the analog of sum \eqref{eq:hahlomaa2} tends to infinity with $n$.
\end{remark}

The goal of the present paper is to prove the sharpest possible converse of this result for doubling metric spaces, and then deduce a few corollaries for specific metric spaces and alternative notions of flatness.

\begin{remark}\label{r:sep-trip}
The definition of  $\beta_\infty$ which is needed for the proof of Theorem \ref{t:Hah} only requires the $\sup$ to be over triples $\{x,y,z\}$ whose mutual distances are $\gtrsim$ the radius of the ball $B$ in question.  It is in this situation that $\partial\{x,y,z\}$ is proportional to the Menger curvature of $\{x,y,z\}$, which is how Hahlomaa had stated his result. This is discussed in more detail see Remark 2.3 in \cite{Sc07_survey} and \cite{Hahlomaa-non-AR} . 
 \end{remark}

\subsection{New results}
The main theorem of this paper, Theorem \ref{thm:upperbound}, is a converse to Theorem \ref{t:Hah} for doubling spaces, to the extent allowed by Remark \ref{rmk:counterex}. To our knowledge, it is the first theorem of this type for rectifiable curves in general doubling metric spaces, i.e., which states that all rectifiable curves in doubling metric spaces admit a quantitative local flatness condition analogous to that in \eqref{e:beta_gamma}.
Further below, we apply Theorem \ref{thm:upperbound} to deduce four corollaries, Corollaries \ref{cor:GHTST}, \ref{cor:ellinfty},  \ref{cor:Banach} and \ref{cor:Hcurves}.

\begin{maintheorem}\label{thm:upperbound}
Let $\Gamma$ be a connected, doubling metric space. Let $\cB$ be a multiresolution family of balls in $\Gamma$, where the inflation factor for the balls of $\cB$ is $A>1$ (see Definition \ref{d:multiresolution}). Then
$$\sum_{B\in \cB} \beta_\infty^\Gamma(B)^p \diam(B) \leq C_p\HH^1(\Gamma)$$
for all $p>2$.
The constant $C_p$ depends only on $p$, the doubling constant of $\Gamma$, and the constant $A$.
\end{maintheorem}

As noted in Remark \ref{rmk:counterex}, this theorem is sharp in the sense that for $p=2$ it is false.

\begin{remark}
The authors conjecture that with similar techniques to those of \cite{Sc07_hilbert} one would get that Theorem \ref{thm:upperbound} holds for non-doubling $\Gamma$ as well.
\end{remark}

We now go on to describe some corollaries of Theorem \ref{thm:upperbound} that will be proven in the paper.

\subsubsection{Gromov-Hausdorff $\beta$-numbers}\label{subsub:GH}
In \cite{DT99}, the authors define another measure of flatness for subsets of arbitrary metric spaces, different from $\bi$. Their notion is essentially a normalized Gromov-Hausdorff distance to Euclidean balls, and applies in all dimensions, not just dimension one. 

(Note: Although for most of the paper $B(z,r)$ will refer to a closed ball in a metric space, for the purposes of agreement with \cite{DT99} in this subsection and in Section \ref{sec:GHTST} we write $B(z,r)$ for an open ball.)

We take the following definitions from \cite[Section 2]{DT99}, specializing to the $1$-dimensional case and making some minor changes to the notation. Let $(M,d)$ be a metric space,  $B(z,r)$ a ball in $M$, and consider (not necessarily continuous) mappings
$$ I: B(z,r)\rightarrow (-r,r)\subseteq \RR.$$
For such a mapping $I$, let
\begin{equation}\label{e:alpha}
\alpha(I) = \epsilon(I) + \delta(I)
\end{equation}

where
$$ \epsilon(I) = \sup\{ ||I(x)-I(y)|-d(x,y)|: x,y\in B(z,r)\},$$
and
$$ \delta(I) = \sup\{ \dist(u,I(B(z,r))): u\in (-r,r)\}.$$
Lastly, we set
$$ \alpha(B(z,r)) =r^{-1} \inf_I \alpha(I),$$
where the infimum is taken over all mappings $I:B(z,r)\rightarrow (-r,r)$.

In \cite{DT99}, the smallness or summability of (the $n$-dimensional version of) $\alpha(B)$ for a metric space is taken as an assumption that is then used to construct interesting embeddings into low-dimensional Euclidean spaces. By contrast, we obtain the summability of $\alpha(B)^{2+\epsilon}$ as a necessary condition for rectifiable curves in doubling metric spaces.

Let $\cB_0 \subseteq \cB$ is the collection of balls in $\cB$ with diameters at least one tenth that of $\Gamma$.
The following corollary is proven in Section \ref{sec:GHTST}.
\begin{maincor}\label{cor:GHTST}
Let $\Gamma$ be a doubling curve in a metric space, with a multiresolution family of balls $\cB$ having inflation factor $A\geq 10$. Then
$$ \sum_{B\in \cB\setminus \cB_0} \alpha(B)^p \rad(B) \lesssim \HH^1(\Gamma),$$
for all $p>2$.
The implied constant depends only on $p$, $A$, and the doubling constant of $\Gamma$. 
\end{maincor}

Corollary \ref{cor:GHTST} can be viewed as a bilateral analog to the upper bound in Jones' Analyst's traveling salesman theorem for arbitrary doubling metric curves, though with a non-sharp exponent.

Note that, while the $\frac{1}{10}$ in the definition of $\cB_0$ is somewhat arbitrary, some restriction in Corollary \ref{cor:GHTST} to ``small'' balls is necessary, since for any ball $B$ in $\cB$ with $\diam(B)\geq 2\diam(\Gamma)$, one has $\alpha(B)\gtrsim 1$.

\subsubsection{$\beta$-numbers for nets in $\ell_\infty$}\label{subsub:ellinfty}
We now turn our attention to the  Banach space $\ell_\infty$, the space of real sequences $(a^1, a^2, \dots)$, equipped with the norm
$$ \|(a^1, a^2, \dots)\| = \sup |a^i|.$$

One may measure the flatness of a subset $S\subseteq \ell_\infty$ in a variety of ways. A method with a clear geometric picture associated to it is to ask: how easy is it to approximate $S$, or a finite net in $S$, by a geodesic in $\ell_\infty$? (Note that the set of geodesics in $\ell_\infty$ a strictly larger class than the class of all \textit{lines} in $\ell_\infty$.) This yields a notion of $\beta$-number that we investigate. 

If $K$ is a set in $\ell_\infty$, $\{X_n\}$ is a family of $2^{-n}$ separated nets in $K$, and $B\in \cB$ is a ball at scale $2^{-n}$ in the associated multiresolution family (see Definition\ref{d:multiresolution}), we will write
\begin{equation}\label{eq:betaellinfty}
\beta^{K,\text{ net}}_{\ell_\infty}(B) \rad(B) = \inf_{L} \sup_{x\in X_{n+1}\cap B} \dist(x,L),
\end{equation}
where 
$$ \dist(x,L) = \inf\{d(x,y):y\in L\}$$
and the infimum is taken over all \textit{geodesics} $L$ in $\ell_\infty$.
Thus, $\beta^{K,\text{ net}}_{\ell_\infty}(B)$ measures how close a net in $B\cap K$ is to a geodesic in $\ell_\infty$. 
The following corollary is proven in Section \ref{sec:ellinfty}.

\begin{maincor}\label{cor:ellinfty}
Let $\Gamma$ be a doubling curve in $\ell_\infty$ with a multiresolution family $\cB$ of balls. Then
$$ \sum_{B\in\cB} \beta^{\Gamma,\text{ net}}_{\ell_\infty}(B)^p \diam(B) \lesssim \HH^1(\Gamma) $$
for all $p>1$.
The implied constant depends only on $p$ and the doubling constant of $\Gamma$.
\end{maincor}

\begin{remark}\label{rmk:ellinfty}
Because of the Kuratowski embedding theorem (see \cite[Section 12.3]{He01}), every separable metric space admits an isometric embedding into $\ell_\infty$. Thus, Corollary \ref{cor:ellinfty} allows one to show that, for a doubling curve $\Gamma$ in an arbitrary metric space, most balls $B$ in a multiresolution family on $\Gamma$ have the property that a net in $B$ is close to lying on an abstract geodesic.

Another way to view Corollary \ref{cor:ellinfty} is as follows: Suppose one has a separable metric space $X$ with a natural class of geodesics, with respect to which one defines a $\beta$ number or $\beta^{\text{net}}$ number and proves an upper bound on the summability of these quantities for curves.
As we have seen, $X$ may be a Euclidean space, a Banach space, or the Heisenberg group, in which case the appropriate summability power $p$ is larger than $1$ and depends on the geometry of the space. 
One may instead consider isometrically embedding the space $X$ in $\ell_\infty$ (using the Kuratowski embedding theorem) and using the richer class of $\ell_\infty$ geodesics to approximate the net points on a given curve in $X$. Corollary \ref{cor:ellinfty} shows that one may achieve better summability (any power greater than $1$) for this quantity. 
In the case of the Heisenberg group, this idea is explored further in Corollary \ref{cor:Hcurves} and Section \ref{sec:Heisenberg}.
\end{remark}

\subsubsection{$\beta$-numbers in uniformly convex Banach spaces.}\label{subsub:banach}
We now turn our attention to a class of Banach spaces that excludes $\ell_\infty$, the \textit{uniformly convex} Banach spaces, which we define below. In these spaces, as in Hilbert spaces, it is natural to simply measure flatness by distance to lines.

If $X$ is a Banach space, $E\subset X$, and $B=B(z,r)$ a ball in $X$, set
\begin{equation}\label{e:betaX}
 \beta_X^E(B)r = \inf_L \sup\{\dist(x,L) : x\in E\cap B\},
\end{equation}
where 
$$ \dist(x,L) = \inf\{d(x,y):y\in L\}$$
and the infimum is taken over all \textit{lines} $L$ in $X$. Note that while in Section \ref{subsub:ellinfty} we allowed all geodesics in our $\beta$-number, here we allow only lines. However, our focus in this section is on Banach spaces in which these notions agree.

We will use the notion of modulus of convexity of a Banach space to connect $\bx$ and $\bi$.
\begin{definition}
The \textit{modulus of convexity} of a Banach space $(X,\|\cdot\|)$ is the function $\delta$ defined as follows:
\begin{equation}\label{eq:convexity}
 \delta(\epsilon) = \inf\left\{ 1 - \left\| \frac{x+y}{2}\right\| : \|x\|=\|y\|=1 \text{ and } \|x-y\|\geq \epsilon\right\}
\end{equation}
for $\epsilon \in [0,2]$.
\end{definition}
For more background and information on this concept, we refer the reader to \cite[Section 1.e]{LT79}. It will be convenient to note that this definition is unchanged if in the set on the right hand side one allows $\|x\|\leq 1$ and $\|y\|\leq 1$. (See \cite[Section 1.e]{LT79}.)

The Banach space $X$ is \textit{uniformly convex} if and only if $\delta(\epsilon)>0$ for all $\epsilon>0$. In this case, a theorem of Pisier \cite{Pi75} states that there is an equivalent norm on $X$ such that 
\begin{equation}\label{eq:powerconvex}
\delta(\epsilon)\geq c\epsilon^q
\end{equation}
for some $c>0$ and $q\geq 2$.

For $p\in (1,\infty)$, the standard $L_p$ spaces each have modulus of convexity satisfying \eqref{eq:powerconvex} with some constant $c=c_p>0$ and exponent
\[ q =  \begin{cases} 
      2 & \text{ if } 1<p\leq 2, \\
      p & \text{ if } p>2. 
   \end{cases}
\]

The following corollary is proven in Section \ref{sec:Banach}.

\begin{maincor}\label{cor:Banach}
Let $(X,\|\cdot\|)$ be a Banach space with modulus of convexity $\delta$ satisfying \eqref{eq:powerconvex} with $c>0$ and $q\geq 2$.  Let $p>q$. 
Then for any doubling curve $\Gamma\subset X$ with a multiresolution family $\cB=\cB^{\Gamma}$ of balls, we have
$$ \sum_{B\in \cB} \beta_X^{\Gamma}(B)^{p}\diam(B) \lesssim \HH^1(\Gamma).$$
The implied constant depends only on $p$, $c$, $q$, the inflation factor of $\cB$, and the doubling constant of $\Gamma$.
\end{maincor}

\begin{remark}
In \cite{ENV18}, Edelen-Naber-Valtorta study the converse problem of finding a curve (or $k$-rectifiable set more generally) that captures a large portion of a given set or measure $\mu$ in a Banach space $X$. (This can be seen as an analog of Theorem \ref{t:TST}\eqref{e:betaE}, where Corollary \ref{cor:Banach} corresponds to Theorem \ref{t:TST}\eqref{e:beta_gamma}.)

The notion of $\beta$-number studied there is an average (rather than supremum) associated to a measure $\mu$ on $X$:
\begin{equation}\label{e:betaENV}
\beta_\mu^1(B(x,r))^2 = \inf_{\text{lines } L} \int_{B(x,r)} \dist(\cdot ,L)^2 \,d\mu.
\end{equation}

Edelen-Naber-Valtorta show that for constructing  curves (the case $k=1$ in \cite{ENV18}) in a Banach space $X$, the relevant property of $X$ is not uniform convexity but uniform smoothness. Without giving the precise definition, we recall that a Banach space is \textit{uniformly smooth} if and only if its dual is uniformly convex \cite{LT79}[Section 1.e]. In that case, as for convexity, there is an associated exponent of smoothness $\alpha$, which is $p$ for $L_p$ if $1<p\leq 2$ and $2$ for $L_p$ if $2<p<\infty$.

In Theorem 2.6 of \cite{ENV18}, the authors prove that, given a measure $\mu$ with positive and finite upper and lower $1$-dimensional densities on a Banach space $X$ with smoothness exponent $\alpha$, integral control on the quantity $\beta^1_\mu(B(x,r))^\alpha$ implies $1$-rectifiability of $\mu$. We refer the reader to \cite{ENV18} for the details and other related results.
\end{remark}

\subsubsection{A comparison with the Heisenberg group}\label{subsub:heis}
We present one more corollary of our work by observing a consequence of Corollary \ref{cor:ellinfty} for curves in the Heisenberg group $\bH$, equipped with its sub-Riemannian Carnot-Carath\'eodory metric $d$. (For a brief introduction to the Heisenberg group and this metric, we refer the reader to \cite{LS14, LS15} and the references in those papers.)

In \cite{FFP07, LS14, LS15}, a notion of $\beta$-number was defined for subsets $K$ of the Heisenberg group as follows. If $B$ is a ball in $\bH$, write
\begin{equation}\label{e:betaH}
 \beta^K_{\bH}(B) \rad(B) = \inf_{L} \sup_{x\in K\cap B} \dist(x,L),
\end{equation}
where
$$\dist(x,L) = \inf\{d(x,y):y\in L\}$$
and the infimum is taken over all so-called \textit{horizontal lines} $L$ in $\bH$. Sub-segments of horizontal lines form a proper sub-class of all geodesics in the Heisenberg group.

By the Kuratowski embedding theorem, we may fix an isometric embedding $\iota\colon \bH \rightarrow \ell_\infty$ and thus view $\bH$ as a subset of $\ell_\infty$. (A different choice of embedding $\iota$ will not affect the results described here, so we suppress it from the notation.)

Given $\iota$, we may view a curve $\Gamma$ in $\bH$ as also being a subset of $\ell_\infty$. Thus, we may compare $\beta_{\bH}$ to the notion of $\beta^{\Gamma,\text{ net}}_{\ell_\infty}$ introduced earlier. Since horizontal lines in $\bH$ are geodesics, we have $\beta^{K}_{\bH}(B) \geq \beta^{K,\text{ net}}_{\ell_\infty}(B).$

Improving earlier work by Ferrari-Franchi-Pajot \cite{FFP07}, Li and Schul \cite{LS15} proved the following analog of \eqref{e:betaE} in Theorem \ref{t:TST}:
\begin{theorem}[\cite{LS15}, Theorem A]\label{thm:LS}
Let $r<4$ be fixed. There is a constant $C_r>0$ such that, for any $K\subseteq \bH$ and multi-resolution family of balls $\cB$ in $K$, if
$$ \diam(K) + \sum_{B\in\cB} \beta^{K}_{\bH}(B)^r\diam(B) < \infty,$$
then there is a rectifiable curve $\Gamma \subseteq E$ such that 
$$ \HH^1(\Gamma) \leq C_r\left( \diam(K) + \sum_{B\in\cB} \beta^{K}_{\bH}(B)^r\diam(B)  \right).$$
\end{theorem}

If one starts with a curve $\Gamma$ whose length is a definite factor larger than its diameter (equation \eqref{eq:longHcurve} below), and one considers the measure $\cM_r$ on the multiresolution family $\cB$ defined by 
$$ \cM_r(\cB') = \sum_{B\in\cB'} \beta^{K}_{\bH}(B)^r\diam(B) $$
for $\cB'\subseteq \cB$, then Theorem \ref{thm:LS} implies that the mass of the whole collection $\cB$ is bounded below by a multiple of the length of $\Gamma$:
$$ \cM_r(\cB) \gtrsim_r \HH^1(\Gamma).$$

The gap between the exponents $p=1+\epsilon$ in Corollary \ref{cor:ellinfty} and $r=4-\epsilon$ in Theorem \ref{thm:LS} then leads to the following observation, which says that for $q<r<4$ and ``most'' balls in $\cB$ (measured with respect to $\cM_r$), we have that 
$$\beta^{\Gamma,\text{ net}}_{\ell_\infty}(B) \lesssim \beta^{\Gamma}_{\bH}(B)^{q}.$$

\begin{maincor}\label{cor:Hcurves}
Let $\Gamma$ be a curve in the Heisenberg group whose length is sufficiently large compared with its diameter (see  \eqref{eq:longHcurve}), with multiresolution family $\cB$ having inflation factor $A=10$. Let  $q<4$,  $r\in (q, 4)$, and $\delta>0$. There is a constant $c=c_{q,r,\delta}>0$ such that if
\begin{equation}
\cB_{c,q} = \{ B\in\cB: \beta^{\Gamma,\text{ net}}_{\ell_\infty}(B) > c\beta^{\Gamma}_{\bH}(B)^q \}
\end{equation}
then
$$ \cM_r(\cB_{c,q}) \leq \delta \HH^1(\Gamma) \leq 2\delta C_r \cM_r(\cB).$$
\end{maincor}
Corollary \ref{cor:Hcurves} indicates that the complement of $\cB_{c,q}$ is a large set in the sense of $\cM_r$. For balls in this complement, $\beta^{\Gamma,\text{ net}}_{\ell_\infty}(B)$ is either itself close to one, or much smaller than $\beta^{\Gamma}_{\bH}(B)$.
We note that 
Theorem I in \cite{LS14} shows that $\cM_4(\cB)\lesssim \HH^1(\Gamma)$.

In Section \ref{sec:Heisenberg}, we present the short proof of Corollary \ref{cor:Hcurves}, and describe how, although one can prove an analogous result in Euclidean spaces, the situations do indicate a genuine difference between the Heisenberg group and Euclidean space.

\begin{remark}
In the setting of the Heisenberg groups   there is also a connection of $\beta$ with singular integrals. An example of a recent result on this is \cite{Vasilis-Sean-2017}. See also  \cite{Vasilis-Sean-Scott} for more general Carnot groups.
\end{remark}

\subsection{Structure of the paper}

In Section \ref{sec:notation}, we establish the basic notations and definitions used in the paper. The proof of Theorem \ref{thm:upperbound} occupies Sections \ref{sec:cubesfiltrations}, \ref{sec:nonflat}, and \ref{sec:flatarcsballs}. 

More specifically, in Section 3, we divide the curve $\Gamma$ into a family of ``cubes'' and consider collections of cubes with $\bi \approx 2^{-M/2}$ for each $M\in\mathbb{N}$. For each $M$, we associate to this collection of cubes a finite number of filtrations of $\Gamma$ into arcs. We then use these filtrations to separate each ball $B\in\cB$ into one of two categories (flat or non-flat), by comparing the flatness of arcs contained in $B$ to $\bi(B)^2\diam(B)$.
More specifically, in Section 3, we divide the curve $\Gamma$ into a family of ``cubes'' and consider collections of cubes with $\bi \approx 2^{-M/2}$ for each $M\in\mathbb{N}$. For each $M$, we associate to this collection of cubes a finite number of filtrations of $\Gamma$ into arcs. We then use these filtrations to separate each ball $B\in\cB$ into one of two categories (flat or non-flat), by comparing the flatness of arcs contained in $B$ to $\bi(B)^2\diam(B)$.

The sum of $\bi(B)^p\diam(B)$ over non-flat balls is then controlled in Section \ref{sec:nonflat} by reducing to a sum over filtrations. The sum over flat balls is controlled in Section \ref{sec:flatarcsballs} using a martingale argument, a modification of the one which appears in \cite{Sc07_hilbert, Sc07_metric}.

Corollaries \ref{cor:GHTST}, \ref{cor:ellinfty}, \ref{cor:Banach}, and \ref{cor:Hcurves} are then proven in Sections \ref{sec:GHTST}, \ref{sec:ellinfty}, \ref{sec:Banach}, and \ref{sec:Heisenberg}, respectively.

\subsection{A table of $\beta$-numbers}
Throughout the paper we define or reference various different notions of flatness for a set. As these may be confusing to keep track of, we provide the following table of reference:

\begin{spacing}{2}
\begin{center}
\begin{tabular}{|c|c|c|}
\hline
Notation															&     Description                  &     Location of definition                  \\
\hline
\hline
$\beta_{\infty}^{\rm Euc}$						&  Distance to lines in $\bR^n$															& \eqref{e:euclidean-beta}\\
\hline
$\del_1$															&  Ordered triangle inequality deficit											& \eqref{e:del1}\\
\hline
$\del$																&  Unordered triangle inequality deficit										& \eqref{e:del}\\
\hline
$\bi$																	&  Supremum of $\del$																				& \eqref{e:bi}\\
\hline
$\alpha$															&  ``Gromov-Hausdorff'' bilateral distance to line segment	& \eqref{e:alpha}\\
\hline
$\beta^{\text{net}}_{\ell_\infty}$		&  Distance of net to geodesics in $\ell_\infty$						& \eqref{eq:betaellinfty}\\
\hline
$\beta_X$															&  Distance to lines in Banach space $X$										& \eqref{e:betaX}\\
\hline
$\beta_\mu^1$													&  Average distance of a Banach space measure to lines			& \eqref{e:betaENV}\\
\hline
$\bti$																&  Supremum of $\del_1$ along an arc												& \eqref{e:bti}\\
\hline                   
\end{tabular}
\end{center}
\end{spacing}

\subsubsection*{Acknowledgments}
The authors would like to thank Jonas Azzam for helpful discussions.
G.~ C.~ David was partially supported by the National Science Foundation under Grant No. DMS-1758709. R.~ Schul was partially supported by the National Science Foundation under Grants No. DMS-1763973 and DMS-1361473

\section{Notation and definitions}\label{sec:notation}
\subsection{Balls, nets, and multiresolution families}
For a metric space $M$, we denote balls in $X$ by 
$$ B(x,r) = \{y\in X: d(x,y)\leq r\}.$$
A ball is considered to be equipped with a center and radius (which may not be uniquely defined by the ball seen only as a set). The radius of $B$ will sometimes be denoted $\rad(B)$.

If $B=B(x,r)$ and $\lambda>0$, then we write
$$ \lambda B = B(x,\lambda r).$$

If $E\subseteq M$, we say that $X\subset E$ is an $\epsilon$-net (or $\epsilon$-separated net) for $E$ if
\begin{enumerate}[(i)]
\item for all $x_1,x_2\in X$ we have $\dist(x_1,x_2) > \epsilon$
\item for all $y\in E$ there exists $x\in X$ such that $\dist(x,y)\leq \epsilon$
\end{enumerate}
Hence $E\subset \bigcup_{x\in X}B(x,\epsilon)$, given an $\epsilon$-net $X$ for $E$.  

Fix a set $E$. Denote by $X^E_n$ a  sequence of nested $2^{-n}$-nets for $E$. In other words, $X^E_n$ are $2^{-n}$-nets for $E$ such that $X^E_{n+1}\subseteq X^E_n$ for each $n\in\bZ$.
\begin{definition}\label{d:multiresolution}
 A multiresolution $\cB$ for a set $E$ (denoted by $\cG^E$ when not clear from context) is defined by
\begin{gather}\label{29_11_05}
\cB^E=\{B(x,A2^{-n}):x\in X^E_n, n \text{ an integer}\}
\end{gather}
for a constant $A>1$.  
The constant $A$ may then be referred to as the \emph{inflation factor of $\cB$}.
\end{definition}

\begin{remark}
Throughout the proof of Theorem \ref{thm:upperbound} in sections \ref{sec:cubesfiltrations} through \ref{sec:flatarcsballs}, most statements will involve constants that depend on the inflation factor $A$ in the given multiresolution family $\cB$ of the given curve $\Gamma$. To avoid repetition, we will not remark on this dependence each time, though of course it is noted in the statement of Theorem \ref{thm:upperbound}.
\end{remark}

\subsection{Curves and sub-arcs}\label{subsec:curve}
Fix a compact, connected set $\Gamma$ in a metric space with a doubling metric $d$. Without loss of generality, when proving Theorem \ref{thm:upperbound}, we may assume that 
$\HH^1(\Gamma)\leq 1$.

In that case, there is a $2$-Lipschitz (not necessarily injective) parametrization $\gamma:\bT\rightarrow\Gamma$, where $\bT=\bR/ \bZ$, i.e., $[0,1]$ with $0$ and $1$ identified. (In \cite{AO17}, the statement with domain $[0,1]$  is attributed to Wa\.zewski \cite{Waz}. A proof with domain $\bT$ can be found in \cite[Proposition 5.1]{RR18} or, with a worse Lipschitz constant, in \cite[Lemma 4.2]{Sc07_metric}.)

By scaling the metric on $\Gamma$, we may assume that $\gamma$ is an arc-length (in particular, $1$-Lipschitz) parametrization. Note that this implies that $\diam(\Gamma)\leq 1$.

An \textit{arc} in $\Gamma$ is the restriction $\gamma|_I$ of $\gamma$ to a compact, connected subset $I\subseteq \bT$. We denote by $\length(\tau)$ the arc-length of $\tau$ (which is simply the length of $I$ as $\gamma$ is an arc-length parametrization) and by $\diam(\tau)$ the diameter of the \textbf{image} of $\tau$, i.e., $\diam(\gamma(I))$.

\subsection{$\beta$-numbers}\label{sec:beta}
Let $M$ be a metric space.

As in \cite{Sc07_metric}, for an ordered triple $(x_1,x_2,x_3)\in M^3$ we define
\begin{equation}\label{e:del1}
\del_1(x_1,x_2,x_3):=\dist(x_1,x_2)+\dist(x_2,x_3)-\dist(x_1,x_3).
\end{equation}
We also define an unordered version of this quantity. Let $\{x_1,x_2,x_3\}\subset M$ be  an unordered triple, and set  
\begin{equation}\label{e:del}
\del(x_1,x_2,x_3)=\min\limits_{\sigma\in S_3}\del_1(x_{\sigma(1)},x_{\sigma(2)},x_{\sigma(3)})\,,
\end{equation}
where $S_3$ is the permutation group on $\{1,2,3\}$. Equivalently,
\begin{gather*}
\del(x_1,x_2,x_3):=\del_1(x_1,x_2,x_3),
\end{gather*}
whenever $\dist(x_1,x_2)\leq\dist(x_2,x_3)\leq\dist(x_1,x_3)$.
We have for all $\{x,y,z\}\subset M$
\begin{gather*}
\del(x,y,z)\leq \diam\{x,y,z\},
\end{gather*}
as well as
\begin{gather*}
0\leq\del(x,y,z)\leq\del_1(x,y,z)\leq 2\diam\{x,y,z\},
\end{gather*}
where non-negativity follows from the triangle inequality.

Let $E$ be a metric space.
Let $B$ be a ball of radius $r$.
We define
\begin{equation}\label{e:bi}
\beta_\infty^E(B)^2 = r^{-1}\sup\{ \del(x,y,z) : x,y,z\in E \cap B\}
\end{equation}
If the  $E$ is understood, we suppress it from the notation and write simply $\bi(B)$.
See the introduction, \cite{Hahlomaa-non-AR}, and \cite{Sc07_survey} for further background on this definition and how it relates to Jones's classical definition in \cite{Jones-TSP}.

We also define an ordered version of $\bi$ for arcs in the parametrization. For an interval $I$ and arc $\tau=\gamma|_I\subseteq \Gamma$, let
\begin{equation}\label{e:bti}
\bti(\tau)^2\diam(\tau) = \sup\{ \del_1(\gamma(a),\gamma(b),\gamma(c)) : a<b<c\in I\}.
\end{equation}

\section{Cubes, filtrations, and flat versus non-flat balls}\label{sec:cubesfiltrations}
We now begin proving Theorem \ref{thm:upperbound} in earnest. 

Let $\Gamma$ be a doubling metric curve and let $\cB$ be a multiresolution family for $\Gamma$, as defined above. As remarked in subsection \ref{subsec:curve}, we without loss of generality equip $\Gamma$ with a $1$-Lipschitz parametrization $\gamma:\bT\rightarrow \Gamma$.

Fix a small absolute constant $\eb>0$ and a large constant $K\in\mathbb{N}$ such that $2^{-K}\leq \eb^2/100$. These will be defined to be sufficiently small in the course of the proof of Theorem \ref{thm:upperbound}. 

\subsection{Cubes}\label{subsec:cubes}
We first split our multiresolution family into a fixed number of disjoint subcollections, using the following lemma from \cite{Sc07_metric}.

\begin{lemma}[Lemma 2.14 of \cite{Sc07_metric}]\label{lemma:separate}
Let $R>0$ be given. There is a $P_1=P_1(R)$ such that one can write a disjoint union
$$ \cB = \cB^1 \cup \dots \cup \cB^{P_1},$$
such that, for each $1\leq p \leq P$, if $B_1, B_2 \in \cB^{p}$ have the same radius $r$, then
$$ \dist(B_1, B_2) \geq R r.$$
The number $P_1$ depends only on $R$ and the doubling constant of $\Gamma$.
\end{lemma}
Note that the proof of this lemma in \cite{Sc07_metric} relies only on the doubling property of $\Gamma$, and not its Ahlfors regularity.

Fix $R>10$ sufficiently large depending on $K$, to be determined after Lemma \ref{lemma:cubes} below, and apply Lemma \ref{lemma:separate} to obtain the disjoint decomposition
$$ \cB = \cB^1 \cup \dots \cup \cB^{P_1}.$$

Given $M\in\mathbb{N}$ and $1\leq p_1\leq P_1$, let 
$$\cB^{p_1}_M=\{B\in \cB^{p_1}: \frac{1}{2}\bi(B)^2\in [2^{-M}, 2^{-M+1})\}.$$
We then split this collection further into
$$\cB^{p_1}_M= \cB^{p_1}_{M,1}\cup....\cup \cB^{p_1}_{M,KM}$$
as follows.
For $1\leq i \leq KM$, set
$$ \cB^{p_1}_{M, i} = \{B\in \cB^{p_1}_M: r(B) = A 2^{-nKM + i}, n\in\mathbb{Z}\}$$
Note that if if $B_1, B_2\in \cB^{p_1}_{M, i}$ and have different radii $r_1> r_2$, respectively, then
$$r_2\leq 2^{-KM}r_1 \leq 2^{-K}r_1\,.$$

\begin{lemma}\label{lemma:cubes}
If $R$ is sufficiently large, depending on $K$, then for each $p_1, M, i$ as above, there exists a family $\cQ^{p_1}_{M, i}$ of sets with the following properties:
\begin{enumerate}[(i)] 
\item There is a bijection $Q:\cB^{p_1}_{M, i}\rightarrow \cQ^{p_1}_{M, i}$ such that
$$ 2B \subseteq Q(B) \subseteq (1+4\cdot 2^{-KM})2B.$$
\item If $Q,Q'\in \cQ^{p_1}_{M, i}$, then $Q\cap Q'=\emptyset$, $Q\subseteq Q'$, or $Q'\subseteq Q$.
\item If $B\neq B'\in\cB'$ have the same radius, then $Q(B)$ and $Q(B')$ are disjoint.
\end{enumerate}
\end{lemma}
\begin{proof}
This construction is entirely contained in Proposition 2.15 of \cite{Sc07_metric} and the preceding discussion.
\end{proof}
We call the elements of $\cQ^{p_1}_{M, i}$ constructed by this lemma ``cubes''.
We note that they are simple variants of Christ's cubes \cite{Ch90}.

In the construction of the lemma, we write $B\colon \cQ^{p_1}_{M, i}\rightarrow \cB^{p_1}_{M, i}$ for the inverse of the map $Q$, so that $B(Q)$ denotes the ball in $\cB^{p_1}_{M, i}$ that gave rise to the cube $Q$. If $Q\in\cQ^{p_1}_{M, i}$, we will write $r(Q)$ and $c(Q)$ to denote the radius and center of the ball $B(Q)$, respectively, and we define
$$ \bi(Q) = \bi(B(Q)),$$ 
extending the definition of $\bi$ from balls to cubes.

For each $1\leq p_1 \leq P_1$, set
$$ \cQ^{p_1} = \bigcup_{M=1}^\infty\bigcup_{i=1}^{KM} \cQ^{p_1}_{M, i},$$
$$ \cQ = \bigcup_{p_1=1}^{P_1} \cQ^{p_1}.$$

\subsection{Defining partial filtrations by arcs}

Given a cube $Q\in \cQ$, let
$$\Lambda(Q)=\{\gamma|_I: I\subset \bT,\ I=\textrm{ a connected component of }\gamma^{-1}(Q),\ \gamma(I)\cap B(Q)\neq \emptyset\}$$

In particular, since $2B(Q)\subset Q$, this means that if $\tau\in \Lambda(Q)$, then $\diam(\tau)\geq r(B(Q)) \geq \frac{1}{6}\diam(Q)$, where $\diam(\tau)$ is always though of as the diameter of its {\it image}.

We need to know the following simple fact about our arcs:
\begin{lemma}\label{l:Q-eta}
If $Q,Q'\in \cQ^{p_1}_{M,i}$ and $\eta\in \Lambda(Q)$ and $\eta\in \Lambda(Q')$, then $Q=Q'$.
\end{lemma}
\begin{proof}
Assume that $\eta\in \Lambda(Q)$ and $\eta\in \Lambda(Q')$ but $Q\neq Q'$

The assumption that $\eta$ is in both $\Lambda(Q)$ and $\Lambda(Q')$ implies that $Q\cap Q'\neq \emptyset$, and hence that $Q\subseteq Q'$ or $Q'\subseteq Q$. Assume without loss of generality that $Q\subseteq Q'$.

Let $B= B(Q)$ have radius $r$ and $B'=B(Q')$ have radius $r'$. Note that we cannot have $r=r'$, since if $r=r'$ then $\dist(B,B')>Rr>10r>\diam(Q')$, which is impossible since both $B$ and $B'$ are contained in $Q'$, which has diameter at most $3r$.

Therefore we either have $r<r'$ or $r'<r$. In either case, the ratio between the larger and smaller of $r$ and $r'$ is at least  $2^K > 10$.

On the other hand, $\eta\in \Lambda(Q)$ implies that
$$ r \leq \diam(\eta) \leq 5r$$
and similarly $\eta\in\Lambda(Q')$ implies that
$$ r'\leq \diam(\eta)\leq 5r'.$$

This is a contradiction.

\end{proof}

Next, for each positive integer $M\in\mathbb{N}$ and $1\leq i \leq kM$, we let 
$$\cF^{p_1}_{M,i}=\cup\{\Lambda(Q):\ Q\in \cQ^{p_1}_{M,i}\}$$
and endow it with the partial order given by containment. 

We note that we do not expect any $\cF^{p_1}_{M,i}$ to cover all of $\Gamma$, nor for each ``level'' of $\cF^{p_1}_{M,i}$ to cover the previous ``level''. Thus, we consider $\cF^{p_1}_{M,i}$ a ``partial filtration'' of $\gamma$.

\subsection{Completing the partial filtrations}

Our next goal is to complete each partial filtration $\cF^{p_1}_{M,i}$ to a ``full'' filtration of the portion of $\Gamma$ covered by its maximal elements. In other words, we will define a collection  $\widehat{\cF^{p_1}_{M,i}}$ of arcs in $\Gamma$ such that:
\begin{itemize}
\item For each $k\geq 0$, the collection of arcs in $\widehat{\cF^{p_1}_{M,i}}$ with exactly $k$ ancestors is a disjoint (up to endpoints) cover of the union of maximal elements of $\cF^{p_1}_{M,i}$. 
\item $ \widehat{\cF^{p_1}_{M,i}} \supseteq \cF^{p_1}_{M,i}$
\end{itemize}

We will perform this completion by starting with $\cF^{p_1}_{M,i}$ and adding subarcs of $\bT$ to our filtration in a way which we now specify. The construction below has the following property: if $\tau_2\subset \tau_1$ are both in $\cF^{p_1}_{M,i}$ and there is no arc $\tau\in \cF^{p_1}_{M,i}$ such that  $\tau_2\subset \tau\subset \tau_1$, then there will be no such arc $\tau \in \widehat{\cF^{p_1}_{M,i}}$ either. 

For each arc $\tau\in \widehat{\cF^{p_1}_{M,i}}$, we will define $\children(\tau)$ to be the maximal elements in $\widehat{\cF^{p_1}_{M,i}}$ contained in $\tau$.
We will write $\children(\tau)=\children_1(\tau)\cup \children_2(\tau)$, a disjoint (up to endpoints) union, and refer to these as {\it type I} and {\it type II} children, where 
$$\children_1(\tau)=\children(\tau)\cap \cF^{p_1}_{M,i}$$ 
and
$$\children_2(\tau)=\children(\tau)\setminus \cF^{p_1}_{M,i}\,.$$ 

The filtration $\widehat{\cF^{p_1}_{M,i}}$ will always have roots in $\cF^{p_1}_{M,i}$.

In order to specify $\widehat{\cF^{p_1}_{M,i}}$, we may simply specify 
$\children(\tau)$ for $\tau \in \widehat{\cF^{p_1}_{M,i}}$.
For  $\tau\notin \cF^{p_1}_{M,i}$ we will always have 
$\children(\tau)= \tau$. (In other words, type $2$ children are never subdivided further.) 
For 
$\tau\in \cF^{p_1}_{M,i}$, the collection $\children_1(\tau)$ is given by maximal $\cF^{p_1}_{M,i}$ elements inside $\tau$, and we need to specify  
$\children_2(\tau)$ such that $\tau$ can be written as a disjoint union $\cup \children(\tau)$.

To this end, we will simply set $\children_2(\tau)$ to be any finite partition of  $\tau \setminus \cup\children_1(\tau)$ into arcs, subject to the condition that
$$ \diam(\eta) \leq  2^{-KM} \diam(\tau) \text{ for all } \eta\in \children_2(\tau).$$
Note that, by our choice of $\cF^{p_1}_{M,i}$, the above inequality is satisfied also when $\eta\in \children_1(\tau)$.

We have now completed our definition of $\children(\tau) = \children_1(\tau) \cup \children_2(\tau)$. The filtration $\widehat{\cF^{p_1}_{M,i}}$ is then defined inductively beginning with the maximal elements of $\cF^{p_1}_{M,i}$ and using the $\children(\tau)$ operation repeatedly.

\subsection{Flat versus Non-flat balls and the two halves of Theorem \ref{thm:upperbound}}\label{subsec:flatnonflat}
In proving Theorem \ref{thm:upperbound}, it is convenient to first dispose of the collection of balls that are too large. 
Let $\cB_0$ be the collection of balls $B\in\cB$ such that $\diam(B) \geq \frac{1}{10}\diam(\Gamma)$. 
Then we have the following:
\begin{lemma}\label{lemma:endpoint}
For each $p>2$,
$$\sum_{B\in \cB_0} \bi(B)^p\diam(B) \lesssim \length(\gamma),$$
where the implied constant depends only on $p>2$ and the doubling constant of $\Gamma$.
\end{lemma}
\begin{proof}
Let $n_0$ be the smallest integer such that $A2^{-n}\geq \frac{1}{10}\diam(\Gamma)$. Each $B\in \cB_0$ is at scale $2^{-n}$ for some $n\leq n_0$. By the doubling property of $\Gamma$, there are at most a fixed number $D$ of balls $B\in \cB_0$ at each such scale. Therefore
\begin{align*}
\sum_{B\in \cB_0} \bi(B)^p\diam(B) &\leq \sum_{B\in \cB_0} \left(\frac{\diam(\Gamma)}{\rad(B)}\right)^{p/2} \diam(B)\\
 &\lesssim \sum_{n=-n_0}^\infty \diam(\Gamma)^{p/2} 2^{n(1-p/2)}\\
&\lesssim \diam(\Gamma)^{p/2} 2^{-n_0(1-p/2)}\\
&\lesssim \diam(\Gamma)\\
&\leq \length(\gamma)
\end{align*}
Note that above we used the fact that $\diam(\Gamma)\lesssim 1$ which we had already assumed without loss of generality in section \ref{subsec:curve}.
\end{proof}

We may now focus on balls $B\in \cB\setminus\cB_0$. Note that for such balls, 
$$ \rad(B) \leq  \diam(B) \leq 2\rad(B).$$

We define two classes of balls in $\cB\setminus \cB_0$ based on the notions defined in this section, calling them colloquiually ``non-flat'' balls and `flat'' balls. Let
$$ \cB_1 = \{B\in \cB\setminus \cB_0: \bti(\tau) > \eb\bi(B) \text{ for some } \tau\in\Lambda(Q(B))\}$$
$$ \cB_2 = \{B\in \cB\setminus \cB_0: \bti(\tau) \leq \eb\bi(B) \text{ for all } \tau\in\Lambda(Q(B))\}$$
For each non-flat ball $B$ with associated $Q=Q(B)$, we will  fix an arc $\tau_Q\in \Lambda(Q)$ which satisfies
$$ \bti(\tau) > \eb\bi(B).$$

We will show the following two propositions.

\begin{proposition}\label{prop:nonflat}
For all $p>2$, 
$$\sum_{B\in\cB_1} \bi(B)^p\diam(B) \lesssim \length(\gamma)$$
The implied constant depends only on $p$ and the doubling constant of $\Gamma$.
\end{proposition}

\begin{proposition}\label{prop:flat}
For all $p>2$, 
$$\sum_{B\in\cB_2} \bi(B)^p\diam(B) \lesssim \length(\gamma).$$
The implied constant depends only on $p$ and the doubling constant of $\Gamma$.
\end{proposition}

These two Propositions, along with Lemma \ref{lemma:endpoint}, combine immediately to prove Theorem \ref{thm:upperbound}.

\section{Non-flat arcs and balls}\label{sec:nonflat}
In this section, we prove Proposition \ref{prop:nonflat}.

We begin with a few lemmas. Recall the definitions of $\del_1$ and $\del$ from subsection \ref{sec:beta}.

\begin{lemma}
Let $\tau$ be an arc in some $\widehat{\cF^{p_1}_{M,i}}$ and let $\children(\tau)$ be the partition of $\tau$ into its children.
 
For each $\eta\in \children(\tau)$ write $a(\eta)$ and $b(\eta)$ for the start and finish of $\eta$ (in the domain of $\tau$).  

Let $a,b,c\in \domain(\tau)$ be three points that are each a start or end of an element in $\children(\tau)$ and such that $a\leq b \leq c$. 
Then 
$$\partial_1(\gamma(a),\gamma(b),\gamma(c))\leq \sum\limits_{\eta\in \children(\tau)}d(\gamma(a(\eta)),\gamma(b(\eta))) - d(\gamma(a(\tau)),\gamma(b(\tau)))$$
\end{lemma}
\begin{proof}
This is a direct application of the triangle inequality.
\end{proof}

Now consider any $Q\in\cQ^{p_1}_{M,i}$ and $\tau=\gamma|_I\in\Lambda(Q)$. If $t_1\leq t_2 \leq t_3$ are in $I$, then, by shifting each point $t_i$ to an endpoint of the arc in $\children(\tau)$ containing it, we may find points  $t'_1 \leq t'_2 \leq t'_3$ such that
$$ d(\gamma(t_i), \gamma(t'_i)) \leq 2^{-K}\bi(Q)^2\diam(Q) \text{ for } i=1,2,3.$$
Therefore, if we write $s(\eta)=\gamma(a(\eta))$ and $f(\eta) = \gamma(b(\eta))$ for the start and end of an arc in the image, we have
\begin{align*}
 \partial_1(\gamma(t_1), \gamma(t_2), \gamma(t_3)) &\leq \partial_1(\gamma(t'_1), \gamma(t_2'), \gamma(t_3')) + 6\cdot 2^{-K}\bi(Q)^2\diam(Q)\\
 &\leq  \sum\limits_{\eta\in \children(\tau)} d(s(\eta),f(\eta)) - d(s(\tau),f(\tau)) + 6\cdot 2^{-K}\bi(Q)^2\diam(Q).
\end{align*}

Hence, we have proven the following lemma:
\begin{lemma}\label{lemma:btibound}
For $Q\in \cQ^{p_1}_{M,i}$ and $\tau\in \Lambda(Q)$, we have
$$ \bti(\tau)^2\diam(\tau) \leq \sum\limits_{\eta\in \children(\tau)}\dist(s(\eta),f(\eta)) - \dist(s(\tau),f(\tau)) + 62^{-K}\bi(Q)^2\diam(Q).$$
If moreover $B(Q)\in \cB_1$ (i.e., is non-flat) and $\tau=\tau_Q$, then 
\begin{equation}\label{e:telescoping-1}
\bti(\tau)^2\diam(\tau)\leq 
2\left(
\sum\limits_{\eta\in \children(\tau)}\dist(s(\eta),f(\eta)) - \dist(s(\tau),f(\tau))
\right)
 \end{equation}

\end{lemma}

Equation \eqref{e:telescoping-1} follows from the fact that 
$$ 62^{-K}\bi(Q)^2\diam(Q) < \frac{1}{2}\bti(\tau)^2\diam(\tau)$$
if $\tau=\tau_Q$, by our choice of $K$ relative to $\eb$ at the start of Section \ref{sec:beta}.

Equation \eqref{e:telescoping-1} will be useful for us as it telescopes well. We can now control the sum of $\bti^2(\tau_Q)\diam(\tau_Q)$ over a single family of non-flat cubes.

\begin{lemma}\label{lemma:onefamily}
For each fixed $M$ and $i$ we have
$$\sum\limits_{Q\in \cQ^{p_1}_{M,i} \atop B(Q)\in \cB_1} \bti(\tau_Q)^2\diam(\tau_Q) \leq 2\length(\gamma).$$
\end{lemma}
\begin{proof}
Using Lemmas \ref{l:Q-eta} and \ref{lemma:btibound}, we write 
\begin{align*}
\sum\limits_{Q\in \cQ^{p_1}_{M,i} \atop B(Q)\in \cB_1} \bti(\tau_Q)^2\diam(\tau)&=
\sum_{\tau \in \cF^{p_1}_{M,i} \atop \tau=\tau_Q,\ B(Q)\in \cB_1}  
      \bti(\tau_Q)^2\diam(\tau)\\
&\leq 2\sum_{\tau \in \cF^{p_1}_{M,i} \atop \tau=\tau_Q,\ B(Q)\in \cB_1}  \left(
\sum\limits_{\eta\in \children(\tau)}d(s(\eta),f(\eta)) - d(s(\tau),f(\tau))
\right)\\
&\leq 2\sum_{\tau \in \widehat{\cF^{p_1}_{M,i}}}  
      \left(
\sum\limits_{\eta\in \children(\tau)}d(s(\eta),f(\eta)) - d(s(\tau),f(\tau))
\right)\\
&\leq 2\length(\gamma)
\end{align*}
as the last sum telescopes and the total length is controlled by the total length of the maximal elements of $\cF^{p_1}_{M,i}$, which is bounded by the length of $\gamma$.
\end{proof}

We now prove Proposition \ref{prop:nonflat}.

\begin{proof}[Proof of Proposition \ref{prop:nonflat}]
It suffices to show that 
$$\sum\limits_{B\in \cB_1} \bi(B)^p\diam(B) \lesssim \length(\gamma)$$
if $p>2$, where the implied constant depends only on $p$, $P_1$, $\eb$, and  $K$.

We write, using the definition of $\cB_1$ and Lemma \ref{lemma:onefamily}:
\begin{align*}
\sum\limits_{B\in\cB_1} \bi(B)^p\diam(B) &= \sum_{p_1=1}^{P_1}\sum_{M=1}^\infty \sum_{i=1}^{KM} \sum\limits_{B\in\cB_1 \atop Q(B) \in \cQ^{p_1}_{M,i}} \bi(B)^p\diam(B)\\
&\leq \sum_{M=1}^\infty 2P_1 KM 2^{-(p-2)(M-1)/2} \epsilon_\beta^{-2}\sum\limits_{Q\in \cQ^{p_1}_{M,i} \atop B(Q)\in \cB_1} \bti(\tau_Q)^2\diam(\tau_Q)\\
&\leq \sum_{M=1}^\infty 4P_1 KM 2^{-(p-2)(M-1)/2} \epsilon_\beta^{-2} \length(\gamma)\\
&\lesssim \length(\gamma),
\end{align*} 
\end{proof}

\section{Flat arcs and balls}\label{sec:flatarcsballs}
In this section, we prove Proposition \ref{prop:flat}.

\subsection{Statement of the key intermediate proposition}\label{sec:219}

For now, fix $p_1\in \{1, \dots, P_1\}$, $M\geq 0$, and $i\in\{1, \dots, KM\}$. Consider the associated collection of non-flat cubes
$$ \Delta = \Delta^{p_1}_{M,i} = \{Q \in \cQ^{p_1}_{M,i}: B(Q)\in \cB_2.\}$$
Recall that these are cubes for which all arcs in $\Lambda(Q)$ are flat.

In proving Proposition \ref{prop:flat}, the main step will be to show the following.

\begin{proposition}\label{prop:219}
There is an absolute constant $c>0$ with the following property. Let $B\in \cB_2 \cap \cB^{p_1}_{M,i}$ with $Q=Q(B)$ the associated cube in $\Delta$. Write
\begin{equation}\label{cubedecomposition}
 Q = R_{Q} \cup \bigcup Q^j
\end{equation}
where $Q^j$ are maximal subcubes of $Q$ in $\Delta$ and $R_{Q}$ is the remainder.

Then
\begin{equation}\label{done219}
\length(R_{Q}) + \sum \diam(Q^j) \geq (1+c\bi(B)^2)\diam(Q).
\end{equation}
\end{proposition}
Here $c$ is fixed after choosing $\eb$. The relationship between $K$ and $\eb$ is also used.

Once we establish Proposition \ref{prop:219}, Proposition \ref{prop:flat} will follow by a martingale argument similar to those in \cite{Sc07_metric, Sc07_hilbert}. This will be done in subsection \ref{sec:propflatproof}.

\subsection{Proof of Proposition \ref{prop:219}}\label{sec:prop219proof}
We continue to use the same notation and assumptions as fixed at the start of subsection \ref{sec:219}: $B\in \cB_2 \cap \cB^{p_1}_{M,i}$, $Q=Q(B)\in \Delta$.

Let $\xi\in\Lambda(B)$ be an arc passing through the center of $B$. 

As in Proposition \ref{prop:219}, we write
$$ Q = R_{Q} \cup \bigcup Q^j, $$
where $Q^j$ are maximal subcubes of $Q$ in $\Delta$ and $R_{Q}$ is the remainder.

Our goal in this subsection is to prove Proposition \ref{prop:219} for $B$.

We will do this by way of the following two lemmas:
\begin{lemma}\label{lemma:220}
For all $C_1\geq 1$, if $\eb$ satisfies $\eb^{-2}>60C_1$, then the following holds:
\begin{equation}\label{eq:lemma220}
\length(R_{Q} \setminus \xi) + \sum_{j: Q^j \cap \xi = \emptyset} \diam(Q^j) \geq C_1 \eb^2\bi(B)^2\diam(Q)
\end{equation}
\end{lemma}

\begin{lemma}\label{lemma:221}
We have
\begin{equation}
\length(R_Q\cap \xi) + \sum_{j: Q^j \cap \xi \neq \emptyset} \diam(Q^j) \geq (1-\epsilon_3\bi(B)^2)\diam(Q)
\end{equation}
where
$$\epsilon_3 = 16\cdot 2^{-K} + \eb^2.$$
\end{lemma}

First we observe that Proposition \ref{prop:219} for $B$ follows from these two lemmas:
\begin{proof}[Proof of Proposition \ref{prop:219}]
Lemmas \ref{lemma:220} and \ref{lemma:221} combine to show that \eqref{done219} holds with
$$ c = C_1\eb^2 - 16 \cdot 2^{-K} - \eb^2.$$
We have chosen $K$ such that $2^{-MK}\leq 2^{-K} <\eb^2$, and $C_1$ can be chosen such that $C_1\geq 20$. This yields \eqref{done219} for the ball $B$ with $c\geq 4\eb^2$, which suffices to prove Proposition \ref{prop:219}.
\end{proof}

Now we work to prove Lemmas \ref{lemma:220} and \ref{lemma:221}.

\begin{proof}[Proof of Lemma \ref{lemma:220}]
To prove Lemma \ref{lemma:220}, it suffices to find a single point $x\in B$ such that
\begin{equation}\label{eq:xifar}
\dist(x,\xi) \geq 3C_1 \eb^2\bi(B)^2\diam(Q)
\end{equation}
Indeed, if $x$ satisfies \eqref{eq:xifar}, then there is an arc $\eta$ containing $x$ of diameter at least $C_1\eb\bi(B)^2\diam(Q)$ whose distance from $\xi$ is at least
$$ C_1 \eb^2\bi(B)^2\diam(Q) \geq 2\cdot 2^{-K} \diam(Q) \geq 2\diam(Q^j) \text{ for all } j. $$
It follows that \eqref{eq:lemma220} holds for $B$.

Suppose, therefore, that there was no such point $x\in B$. In that case, for any points $x_1,x_2,x_3\in B$, we can find $x_1',x_2',x_3'\in \xi$ such that
$$ d(x_i, x_i')\leq 3C_1 \eb\bi(B)^2\diam(Q) \text{ for } i=1,2,3.$$
Hence
\begin{align*}
\del(x_1, x_2, x_3) &\leq 9C_1 \eb^2\bi(B)^2\diam(Q) + \del(x_1',x_2',x_3')\\
&\leq  9C_1 \eb^2\bi(B)^2\diam(Q) + \bti(\xi)^2\diam(\xi)\\
&\leq  54C_1 \eb^2\bi(B)^2\diam(B) +  6\eb^2 \bi(B)^2\diam(B)\\
&\leq  60C_1 \eb^2 \bi(B)^2\diam(B).
\end{align*}
Since $x_i$ were arbitrary in $B$, it follows that
$$ \bi(B)^2\diam(B) \leq 60C_1\eb^2\bi(B)^2\diam(B),$$
which is a contradiction for $\eb^2< 1/(60C_1)$.
\end{proof}

\begin{proof}[Proof of Lemma \ref{lemma:221}]
By assumption, 
$$ 2^{-(M+1)}\leq \frac{1}{2}\bi(B)^2 \leq 2^{-M}.$$ 

Recall that
$$ \diam(Q) \leq (1+4\cdot 2^{-KM})\diam(2B).$$
Write $[a,c]= \domain(\xi)$, so that $O_1:=\gamma(a)$ and $O_2:=\gamma(c)$ are in $\partial Q$. (Note that $\gamma$ must both enter and exit $Q$, since $B(Q)\notin \cB_0$.) There is also $b\in (a,c)$ such that $O:=\gamma(b)$ is the center of $B$.

Hence,
\begin{align*}
d(O_1, O_2) &= d(O_1, O) + d(O,O_2) - \del_1(a,b,c)\\
&\geq 2\rad(2B) - \bti(\xi)^2\diam(\xi)\\
&\geq \diam(Q) - 4\cdot 2^{-KM}\diam(Q) - \eb^2 \bi(B)^2\diam(Q)\\
&= (1- 4 \cdot 2^{-KM} - \eb^2 \bi(B)^2)\diam(Q)\\
&\geq (1 - (8\cdot 2^{-(K-1)M} + \eb^2)\bi(B)^2) \diam(Q).
\end{align*}
We may therefore set
$$ \epsilon_3 = 16 \cdot 2^{-K} + \eb^2  \geq 8\cdot 2^{-(K-1)M} + \eb^2.$$

\end{proof}

\subsection{Proof of Proposition \ref{prop:flat}}\label{sec:propflatproof}
All that remains to prove Theorem \ref{thm:upperbound} is to prove Proposition \ref{prop:flat}.

We begin by summing over a fixed family $\Delta = \Delta^{p_1}_{M,i}$ as defined at the beginning of this section.

\begin{proposition}\label{prop:flatonefamily}
For each $p>2$, we have
$$ \sum_{B:Q(B) \in \Delta} \bi(B)^p \diam(B) \lesssim 2^{-\left(\frac{p}{2}-1\right)M} \length(\gamma)$$
where the implied constant depends only on $p$ and the constant $c$ from Proposition \ref{prop:219}.
\end{proposition}
\begin{proof}
To begin, suppose that $\Delta$ is a finite collection of cubes.

For each cube $Q\in\Delta$, we will construct a weight $w_Q:Q\rightarrow [0,\infty)$ satisfying three conditions:
\begin{enumerate}[(i)]
\item\label{weight1} $\int_Q w_Qd\ell\geq \diam(Q)$
\item\label{weight2} for almost every $x_0\in \Gamma$, 
	$\sum\limits_{Q \in \Delta} w_{Q}(x_0) < C 2^{M}$,\\
	where $C$ is a positive constant depending only on $c$
\item\label{weight3} $\text{supp} (w_Q) \subset Q$.\\
\end{enumerate}
We will construct $w_Q$ as a martingale.
We denote by $w_Q(Z):=\int_Z w_Qd\ell$.
Set
$$ w_Q(Q)=\diam(Q).$$
Assume now that $w_Q(Q')$ is defined.  We define $w_Q(Q'^i)$ and $w_Q(R_{Q'})$,
where
$$ Q'=(\cup  Q'^i) \cup R_{Q'}.$$
is a decomposition as given by equation \eqref{cubedecomposition}.

Take
$$ w_Q(R_{Q'})=\frac{w_Q( Q')}{s'} \length(R_{Q'}) $$
and
$$ w_Q( Q'^i)=\frac{w_Q(Q')}{s'}\diam(Q'^i),$$
where 
$$s'=\length(R_{Q'})+\sum_i \diam(Q'^i).$$
This will give us $w_Q$. Note that $s'\lesssim \length(\Gamma\cap Q')$.
Clearly \eqref{weight1} and \eqref{weight3} are satisfied.
To see \eqref{weight2}:
\begin{align*}
\frac{w_Q( Q'^{i^*})}{\diam(Q'^{i^*}) }
&=
\frac{w_Q( Q')}{s'}\\
&=
\frac{w_Q( Q')}{\diam(Q') }
\frac{\diam(Q' )}{s'}\\
&=
\frac{w_Q( Q')}{\diam(Q') }
\frac{\diam(Q') }
	{\length(R_{Q'}) + 
		\sum\limits_{i} \diam(Q'^i) }\\
&\leq
\frac{w_Q( Q')}{\diam(Q') }
\frac{1}
	{1+ c' 2^{-M}}\\
\end{align*}
for $c'$ depending only on $c$ (the ultimate inequality followed from Proposition \ref{prop:219}.

And so,
$$\frac{w_Q( Q'^{i^*})}{\diam(Q'^{i^*})} \leq 
 	q   \frac{w_Q( Q')}{\diam(Q')}
$$
with $q=\frac{1}{1+ c' 2^{-M}}$.

Now, suppose 
that  $x\in Q_N \subset ...\subset Q_1$.
we  get:
\begin{align*}
\frac{w_{Q_1}(Q_N)}{\diam(Q_N)} &\leq 
  q\frac{w_{Q_1}(Q_{N-1})}{\diam(Q_{N-1})} \\
  &\leq...\\
  &\leq
  q^{N-1}\frac{w_{Q_1}(Q_{1})}{\diam(Q_1)}=q^{N-1}.
\end{align*}

 Hence,  we have $w_{Q_1}(x) \lesssim q^{N-1}$.
This will give us (ii) as a sum of a geometric series,  since
$$\sum q^n = \frac{1}{1-q}\lesssim \frac{1}{2^{-M}}=2^{M}.$$

Now,
\begin{align*}
\sum_{B: Q(B) \in \Delta}\bi(B)^p\diam(B) 
&\lesssim 
2^{-(p/2)M}\sum_{B: Q(B) \in \Delta}\diam(B)\\
&\lesssim
2^{-(p/2)M}\sum_{B: Q(B) \in \Delta}\int w_{Q(B)}(x)d\ell(x)\\
&=
2^{-(p/2)M}\int \sum_{B: Q(B) \in \Delta} w_{Q(B)}(x)d\ell(x)\\
&\lesssim
2^{-(p/2)M}\int  2^{M} d\ell(x)\\
&\lesssim
2^{-\left(\frac{p}{2}-1\right)M}\length(\gamma).
\end{align*}
\end{proof}

Finally, we complete the proof of Proposition \ref{prop:flat} by summing over all $p_1, M,i$:

\begin{proof}[Proof of Proposition \ref{prop:flat}]
We have, for $p>2$,
\begin{align*}
\sum_{B\in \cB_2} \bi(B)^p \diam(B) &= \sum_{p_1=1}^{P_1} \sum_{M=1}^\infty \sum_{i=1}^{KM} \sum_{B:Q(B)\in \Delta^{p_1}_{M,i}} \bi(B)^p\diam(B)\\
&\leq \sum_{M=1}^\infty P_1 KM 2^{-\left(\frac{p}{2}-1\right)M} \length(\gamma)\\
&\lesssim \length(\gamma),
\end{align*}
where the implied constant depends only on $p$ and the doubling constant of $\Gamma$.

To conclude, the case where $\Delta$ is infinite is obtained as a limit, as our bounds do not depend on the cardinality of $\Delta$.
\end{proof}

\section{``Gromov-Hausdorff'' $\beta$ numbers}\label{sec:GHTST}

Recall the definition of $\alpha$ from subsection \ref{subsub:GH}. In this section, we will relate $\alpha$ to $\bi$ for metric curves, and obtain 
Corollary \ref{cor:GHTST} as a corollary of Theorem \ref{thm:upperbound}.

\subsection{Orders and preliminary lemmas}
To prove Corollary \ref{cor:GHTST}, a useful notion is that of an \textit{order} of a set in a metric space, as defined by Hahlomaa in \cite{Hahlomaa-non-AR}.

\begin{definition}\label{def:order}
Let $E$ be a subset of a metric space $M$. An \textit{order}  on $E$ is an injective map $o:E\rightarrow \RR$ such that
$$ o(x)<o(y)<o(z) \Rightarrow d(x,z) >  \max\{d(x,y),d(y,z)\}$$
for all  $x,y,z\in E$.
\end{definition}
Note that if $o$ is an order on $E$ and $x,y,z\in E$ satisfy $o(x)<o(y)<o(z)$, then
$$ \del(x,y,z) = \del_1(x,y,z).$$

A key fact about orders is the following lemma of Hahlomaa.

\begin{lemma}[Lemma 2.3 of \cite{Hahlomaa-non-AR}]\label{lem:Hahlomaa}
Let $K\geq 1$ and $\epsilon>0$. Suppose that $E$ is a metric space such that 
\begin{enumerate}[(i)]
\item\label{hah1} $d(x,y)\leq Kd(z,w)$ for all $x,y,z,w\in E$ with $z\neq w$,
\item\label{hah2} $ d(x,z) \geq d(x,y) + \epsilon d(y,z)$ whenever $x,y,z\in E$ are such that $d(x,z)=\diam(\{x,y,z\})$, 
\item\label{hah3} $\epsilon^3 \geq \frac{4K-1}{4K+1}$, and
\item\label{hah4} $\sharp E \neq 4$.
\end{enumerate}
Then $E$ has an order.
\end{lemma}

We will now use Lemma \ref{lem:Hahlomaa} to show that nets in balls with sufficiently small $\bi$ can be ordered. Two versions of this result will be useful.

\begin{lemma}\label{lem:order}
Let $\Gamma$ be a compact connected set in a metric space $M$, and let $B=B(p,A \cdot 2^{-n})$ be a ball centered on $\Gamma$. Assume that $A\geq 10$ and
$$ \bi(B)^2 \leq (24A^2(16A+1))^{-1} \leq 1/40.$$

\begin{enumerate}[(i)]
\item If $E$ is a $\eta\diam(B)$-net for $B\cap \Gamma$, for $\eta=2\bi(B)$, then $E$ has an order.
\item If $\diam(B)<\frac{1}{10}\diam(\Gamma)$, $X_{n+1}$ is a $2^{-(n+1)}$-net for $\Gamma$, and $N=X_{n+1}\cap B$, then $N$ has an order.
\end{enumerate}
\end{lemma}
\begin{proof}
We begin with (i). Let 
$$ \eta = 2\bi(B),$$
and let $E$ be a maximal $\eta\diam(B)$-net in $B$. We first work to show that $E$ satisfies the assumptions \ref{hah1}, \ref{hah2}, \ref{hah3}, and \ref{hah4} of Lemma \ref{lem:Hahlomaa}, with appropriate choice of $K$ and $\epsilon$. We verify these assumptions in equations \eqref{eq:hah1}, \eqref{eq:hah2}, \eqref{eq:hah3}, and \eqref{eq:hah4} below.

It follows immediately from the definition of $E$ that if we set $K=1/\eta$, then
\begin{equation}\label{eq:hah1}
 d(x,y) \leq Kd(z,w) \text{ for all } x,y,z,w\in E \text{ with } z\neq w.
\end{equation}
Moreover, if $\{x,y,z\}\subseteq E$ is such that $d(x,z) = \diam\{x,y,z\}$, then
\begin{align*}
d(x,z) &= d(x,y) + d(y,z) - \partial_1(x,y,z)\\
&\geq d(x,y) + d(y,z) - \bi(B)^2\diam(B)\\
&= d(x,y) + d(y,z) - \left(\frac{\eta}{4}\right)\eta\diam(B)\\
&= d(x,y) + d(y,z) - \left(\frac{\eta}{4}\right)d(y,z)\\
&\geq d(x,y) + (1-\frac{\eta}{4})d(y,z).
\end{align*}
Setting $ \epsilon = 1 -\frac{\eta}{4}$, we get
\begin{equation}\label{eq:hah2}
d(x,z) \geq d(x,y) + \epsilon d(y,z).
\end{equation}

Note that 
\begin{equation}\label{eq:hah3}
\epsilon^3 =  \left(1 -\frac{\eta}{4}\right)^3 \geq 1 - \frac{3\eta}{4} \geq 1 - \frac{2\eta}{1+\eta} =  \frac{4K-1}{4K+1}.
\end{equation}

Lastly, since $\eta\leq 1/20$, the number of points $m$ in $E$ must satisfy
$$ \frac{1}{2}\diam(B) \leq \HH^1_\infty(B) \leq m\cdot2\eta\diam(B) \Rightarrow m\geq 5.$$
(Here $\HH^1_\infty$ denotes the Hausdorff content.) In other words, 
\begin{equation}\label{eq:hah4}
\sharp E \geq 5.
\end{equation}

The four equations \eqref{eq:hah1}, \eqref{eq:hah2}, \eqref{eq:hah3}, and \eqref{eq:hah4} verify the assumptions of Lemma \ref{lem:Hahlomaa}, and hence there is an order on $E$.

The proof of statement (ii) in this lemma is quite similar:

As above, we first show that $N$ satisfies the assumptions of Lemma \ref{lem:Hahlomaa}, with appropriate choice of $K$ and $\epsilon$. These assumptions are equations \eqref{eq:hah1'}, \eqref{eq:hah2'}, \eqref{eq:hah3'}, and \eqref{eq:hah4'} below.

If we set $K=4A$, then
\begin{equation}\label{eq:hah1'}
 d(x,y) \leq Kd(z,w) \text{ for all } x,y,z,w\in N \text{ with } z\neq w.
\end{equation}
Moreover, if $\{x,y,z\}\subseteq N$ is such that $d(x,z) = \diam\{x,y,z\}$, then
\begin{align*}
d(x,z) &= d(x,y) + d(y,z) - \partial_1(x,y,z)\\
&\geq d(x,y) + d(y,z) - \bi(B)^2\diam(B)\\
&= d(x,y) + d(y,z) - 4A\bi(B)^2 d(y,z)\\
&\geq d(x,y) + (1-4A\bi(B)^2)d(y,z).
\end{align*}
Setting $ \epsilon = 1-4A\bi(B)^2$, we get
\begin{equation}\label{eq:hah2'}
d(x,z) \geq d(x,y) + \epsilon d(y,z).
\end{equation}

Note that 
\begin{equation}\label{eq:hah3'}
\epsilon^3 \geq  1-12A\bi(B)^2 \geq 1 - \frac{4A}{24A^2(16A+1)} \geq \frac{4K-1}{4K+1}.
\end{equation}

Lastly, since $A\geq 10$ and $B$ does not contain all of $\Gamma$, the number of points $m$ in $N$ must satisfy
\begin{equation}\label{eq:hah4'}
A\cdot 2^{-n} \leq \HH^1_\infty(B) \leq m\cdot 2^{-(n+1)} \Rightarrow m\geq 2A \geq 5.
\end{equation}

The four equations \eqref{eq:hah1'}, \eqref{eq:hah2'}, \eqref{eq:hah3'}, and \eqref{eq:hah4'} verify the assumptions of Lemma  \ref{lem:Hahlomaa}, and hence $N$ has an order.
\end{proof}

\subsection{Proof of Corollary \ref{cor:GHTST}}

We now relate our notions of $\bi$ and $\epsilon(I)$:
\begin{lemma}\label{lemma:betacompare1}
For any curve $\Gamma$ and any ball $B=B(z,r)\subseteq \Gamma$, there is a $1$-Lipschitz mapping $I:B\rightarrow (-r,r)$ such that
$$ \epsilon(I) \leq C\bi(B)\diam(B),$$
where $C$ is an absolute constant.
\end{lemma}
\begin{proof}
We may assume, by taking $C$ sufficiently large, that $\bi(B)\leq (24A^2(16A+1))^{-1}$.

Let $E$ be an $\eta\diam(B)$-net in $B$, where $\eta = 2\bi(B)$. Lemma \ref{lem:order}(i) grants us an order $o$ on $E$. We order the set $E$ according to $o$ and write
$$ E = \{x_1, x_2, \dots, x_n\}$$
with $o(x_i)<o(x_j)$ for all $1\leq i < j \leq n$.

Let $f:B\rightarrow \RR$ be the map
$$ f(x) = d(x_1,x).$$
Note that $f$ is $1$-Lipschitz and therefore maps $B$ into a closed interval of length at most $2r$.

We now claim that $f$ satisfies the condition
\begin{equation}\label{eq:orderE}
||f(x)-f(y)| - d(x,y)|\leq \bi(B)^2\diam(B)
\end{equation}
for all $x,y\in E$. Since $f$ is $1$-Lipschitz, 
$$ |f(x) - f(y)| \leq d(x,y) \text{ for all } x,y\in E.$$
On the other hand, consider $x_i$ and $x_j$ in $E$ with $i<j$. Using the order, we see immediately that $f(x_j) - f(x_i) \geq 0$. In addition,
\begin{align*}
f(x_j) - f(x_i) &= d(x_1, x_j) - d(x_1, x_i)\\
&= d(x_i, x_j) - \partial(x_1, x_i, x_j)\\
&\geq d(x_i, x_j) - \bi(B)^2\diam(B).
\end{align*}
which completes the proof of \eqref{eq:orderE}.

Now consider arbitrary points $x',y'\in B$, not necessarily in $E$. Let $x$ and $y$ be corresponding points of $E$ within distance $\eta\diam(B)$ of $x'$ and $y'$, respectively.

It follows that
\begin{align*}
||f(x') - f(y')| - d(x',y')| &\leq 4\eta\diam(B) + ||f(x)-f(y)|-d(x,y)|\\
&\leq 4\eta\diam(B) + \bi(B)^2\diam(B)\\
&\leq 9\bi(B)\diam(B).
\end{align*}

Lastly, we may postcompose $f$ with a translation so that the center of $B$ maps to $0\in\RR$. This yields a $1$-Lipschitz map $I:B\rightarrow (-r,r)$ such that
$$ \epsilon(I) \leq C\bi(B)\diam(B),$$
as desired.
\end{proof}

For the next lemma, we recall some of the notation used in the proof of Theorem \ref{thm:upperbound}. In particular, suppose we have a doubling curve $\Gamma$ and a $1$-Lipschitz parametrization $\gamma:\bT\rightarrow\Gamma$ as in the start of subsection \ref{subsec:curve}. We will use the notion of $\bti(\tau)$ for an arc $\tau$ of $\gamma$ defined in subsection \ref{sec:beta}. We will also use the ``cube'' decomposition and the division of $\cB$ into families $\cB^{p_1}_{M,i}$ given in subsection \ref{subsec:cubes}, and the distinction between ``non-flat'' balls $\cB_1$ and ``flat'' balls $\cB_2$ given in subsection \ref{subsec:flatnonflat}.

If $B$ is a ball in $\Gamma$, we use the notation $\Lambda_c(B)$ to denote the connected components of $\gamma^{-1}(B)$ whose images contain the center of $B$.

\begin{lemma}\label{lemma:betacompare2}
Let $\Gamma$ be a curve in a metric space and $B=B(z,r)\in \cB\setminus \cB_0$ a ball in the multiresolution family. Let
$$ I :B\rightarrow (-r,r)$$
be the mapping provided by Lemma \ref{lemma:betacompare1}. Lastly, let $\tau\in\Lambda_c(B)$. 

Then
$$ \delta(I) \leq C'( \bti(\tau)^2 r + \bi(B) r),$$
where $C'$ is an absolute constant.
\end{lemma}
\begin{proof}
We parametrize $\Gamma$ by $\gamma$ as in subsection \ref{subsec:curve}.

We first argue, similarly to Lemma \ref{lemma:221}, that
\begin{equation}\label{eq:lengthbound}
\diam(\tau) \geq 2r - \bti(\tau)^2 r.
\end{equation}
for some absolute constant $c>0$. Indeed, let $\tau = \gamma_{[a,b]}$, so that $O_1 = \gamma(a)$ and $O_2=\gamma(b)$ lie on $\partial B$. Let $O=\gamma(c)$ be the center of $B$, for some $c\in (a,b)$.

Then 
\begin{align*}
d(O_1, O_2) &= 2r - \partial_1(O_1,O,O_2)\\
&\geq 2r - \bti(\tau)^2\diam(\tau),
\end{align*}
which proves \eqref{eq:lengthbound}.

It now follows from \eqref{eq:lengthbound}, and the properties of our chosen $I$, that $I(\tau) \subseteq I(B)\subseteq (-r,r)$ contains an interval of length at least 
$$2r - \bti(\tau)^2r - C\bi(B)r.$$
Hence $\delta(I) \leq \bti(\tau)^2 r + C\bti(B) r$.
\end{proof}

\begin{proof}[Proof of Corollary \ref{cor:GHTST}]
We have immediately from Lemmas \ref{lemma:betacompare1}  and \ref{lemma:betacompare2} that
$$ \alpha(B)r \lesssim \bi(B)r + \bti(\tau)^2 r$$
for each $B=B(z,r) \in \cB\setminus \cB_0$ and $\tau\in \Lambda_c(B)$.

We therefore have, for $p>2$, that
\begin{equation}\label{eq:alphasum}
\sum_{B\in\cB\setminus\cB_0} \alpha(B)^p\rad(B) \lesssim \sum_{B\in \cB\setminus \cB_0} \bi(B)^p\diam(B) + \sum_{B\in \cB\setminus \cB_0} \bti(\tau_B)^{2p}\diam(B),
\end{equation}
where $B\mapsto \tau_B$ is any function that maps each ball $B\in\cB\setminus\cB_0$ to an arc $\tau\in\Lambda_c(B)$.

The first sum on the right hand side of \eqref{eq:alphasum} is bounded by $\length(\gamma)$, up to a constant depending only on $p$ and the doubling constant of $\Gamma$, by Theorem \ref{thm:upperbound}. It remains to bound the second sum, which we do using Lemma \ref{lemma:onefamily}, similarly to the proof of Proposition \ref{prop:nonflat}.

Recall the division of $\cB$ into families $\cB^{p_1}_{M,i}$ for $p_1\in \{1,\dots, P_1\}$, $M\in\mathbb{N}$, and $i\in \{1, \dots, KM\}$, and the construction of families of cubes $\cQ_{M,i}$. For each choice of $p_1,  M, i$, let
$$ \cB^{p_1, *}_{M,i} = \{B\in \cB^{p_1}_{M,i} : \bti(\tau)\leq \bi(B) \text{ for all } \tau\in \Lambda_c(B)\},$$
$$ \cB^{p_1, **}_{M,i} = \{B\in \cB^{p_1}_{M,i} : \bti(\tau)>\bi(B) \text{ for some } \tau=\tau_B\in \Lambda_c(B)\},$$
In that case, we control the third term in equation \eqref{eq:alphasum} as follows:
\begin{align}
\sum_{B\in \cB\setminus \cB_0} \bti(\tau_B)^{2p}\diam(B) &= \sum_{p_1=1}^{P_1}\sum_{M=1}^\infty\sum_{i=1}^{KM}\sum_{B\in \cB^{p_1, *}_{M,i}} \bti(\tau_B)^{2p}\diam(B) + \sum_{p_1=1}^{P_1}\sum_{M=1}^\infty\sum_{i=1}^{KM}\sum_{B\in\cB^{p_1, **}_{M,i}} \bti(\tau_B)^{2p}\diam(B)\\
&\leq \sum_{p_1=1}^{P_1}\sum_{M=1}^\infty\sum_{i=1}^{KM}\sum_{B\in \cB^{p_1, *}_{M,i}} \bi(\tau_B)^{2p}\diam(B) + \sum_{p_1=1}^{P_1}\sum_{M=1}^\infty\sum_{i=1}^{KM}\sum_{B\in\cB^{p_1, **}_{M,i}} \bti(\tau_B)^{2p}\diam(B) \label{eq:twosums}
\end{align}
We control the first main sum in \eqref{eq:twosums} simply by Theorem \ref{thm:upperbound}. For the second main sum in \eqref{eq:twosums}, notice that each ball in $B\in\cB^{p_1, **}_{M,i}$ is non-flat, i.e., in $\cB_1$. By Lemma \ref{lemma:onefamily}, we can therefore control each innermost sum in the second main sum of \eqref{eq:twosums} by
$$ \sum_{B\in\cB^{p_1, **}_{M,i}} \bti(\tau_B)^{2p}\diam(B) \lesssim \sum_{Q\in\cQ^{p_1}_{M,i}\atop B(Q)\in \cB_1} \bti(\tau_Q)^{2p}\diam(\tau_Q) \lesssim \length(\gamma),$$
where $\tau_Q$ is an extension of $\tau_B$ to an arc in $\Lambda(Q)$.

Hence, the second main sum in \eqref{eq:twosums} is controlled by
\begin{align*}
\sum_{p_1=1}^{P_1}\sum_{M=1}^\infty\sum_{i=1}^{KM}\sum_{B\in\cB^{p_1, **}_{M,i}} \bti(\tau_B)^{2p}\diam(B) &\leq \sum_{M=1}^\infty 2P_1 KM 2^{-(2p-2)(M-1)/2} \sum\limits_{Q\in \cQ^{p_1}_{M,i} \atop B(Q)\in \cB_1} \bti(\tau_Q)^2\diam(\tau_Q)\\
&\lesssim \sum_{M=1}^\infty P_1 KM 2^{-(2p-2)(M-1)/2} \length(\gamma)\\
&\lesssim \length(\gamma).
\end{align*}
This completes the proof.

\end{proof}

\section{$\beta$ numbers for nets in $\ell_\infty$}\label{sec:ellinfty}

Recall the notion of $\beta^{\Gamma,\text{ net}}_{\ell_\infty}$ from subsection \ref{subsub:ellinfty}.

In the proof of Corollary \ref{cor:ellinfty}, we will use the following notation: We say that an $n$-tuple of points $(x_1,\dots,x_n)$ in $\ell_\infty$ is \textit{ordered} if the map $o(x_i)=i$ is an order in the sense of Hahlomaa. (See Definition \ref{def:order}.) In other words, the $n$-tuple is ordered when
$$ \|x_i - x_k\| > \max\{\|x_i-x_j\|, \|x_j-x_k\|\} \text{ for each } i\leq j \leq k.$$
We say that the $n$-tuple is \textit{$r$-separated}, for some $r\geq 0$, if $\|x_i - x_k\|\geq r$ for each $i,j$.

Corollary \ref{cor:ellinfty} will follow from Lemma \ref{lem:ellinfty} below and our main result, Theorem \ref{thm:upperbound}.

\begin{lemma}\label{lem:ellinfty}
Let $(x_1, x_2, \dots, x_n)$ be an ordered $n$-tuple of $r$-separated points in $\ell_\infty$. 

Assume that 
$$ h = \sup_{i<j<k} \partial_1(x_i, x_j, x_k) < r/200. $$

Then there is a geodesic $L\subset \ell_\infty$ such that
$$ \sup_{1\leq i \leq n} \dist(x_i, L) \leq  15h $$
\end{lemma}

\begin{proof}[Proof of Corollary \ref{cor:ellinfty}]
As in the proof of Theorem \ref{thm:upperbound}, we must first dispose of the balls that are ``too large''.  Namely, let 
$$ \cB_0 = \{B\in \cB: \diam(B) \geq \frac{1}{10}\diam(\Gamma)\}.$$

Let $n_0$ be the smallest integer such that $A2^{-n}\geq \frac{1}{10}\diam(\Gamma)$. Then, as in Lemma \ref{lemma:endpoint}, we have that each $B\in \cB_0$ is at scale $2^{-n}$ for some $n\leq n_0$. By the doubling property of $\Gamma$, there are at most a fixed number $D$ of balls $B\in \cB_0$ at each such scale. Therefore
\begin{align*}
\sum_{B\in \cB_0} \beta^{\Gamma,\text{ net}}_{\ell_\infty}(B)^p\diam(B) &\leq \sum_{B\in \cB_0} \left(\frac{\diam(\Gamma)}{\rad(B)}\right)^{p} \diam(B)\\
 &\lesssim \sum_{n=-n_0}^\infty \diam(\Gamma)^{p} 2^{n(1-p)}\\
&\lesssim \diam(\Gamma)^{p} 2^{-n_0(1-p)}\\
&\lesssim \diam(\Gamma)\\
&\leq \length(\gamma)
\end{align*}

To complete the proof of Corollary \ref{cor:ellinfty}, it suffices to show that for each ball $B\in \cB\setminus\cB_0$ from the multiresolution family of $\Gamma$,
\begin{equation}\label{eq:bibound}
 \beta^{\Gamma,\text{ net}}_{\ell_\infty}(B) \lesssim \bi(B)^2,
\end{equation}
with an absolute implied constant, as we can then apply Theorem \ref{thm:upperbound}.

By adjusting the implied constant in \eqref{eq:bibound}, we may assume that $\bi(B)^2$ is small enough so that Lemma \ref{lem:order} is applicable.

Given a ball $B\in \cB\setminus\cB_0$, write $\{x_1,\dots,x_n\}=X_{n+1}\cap B$, where $X_{n+1}$ is the net at scale $2^{-(n+1)}$ of $\Gamma$.

By Lemma \ref{lem:order}(ii), we may re-number the points so that the $n$-tuple $(x_1, \dots, x_n)$ is ordered. Moreover, it is $r$-separated for $r=2^{-(n+1)} \gtrsim \rad(B)$.

Let 
$$h = \sup_{i<j<k} \del_1(x_i,x_j,x_k) \leq \bi(B)^2 \rad(B).$$

If $h\geq r/200$, then \eqref{eq:bibound} holds automatically, since $r\gtrsim \rad(B)$.

Otherwise, Lemma \ref{lem:ellinfty} implies that
$$ \beta^{\Gamma,\text{ net}}_{\ell_\infty}(B) \lesssim \frac{15h}{\rad(B)} \lesssim \bi(B)^2.$$
\end{proof}

\begin{remark}
Lemma \ref{lem:ellinfty} is false in Euclidean space. For example, consider the three $1$-separated points $x_1=(0,0), x_2=(1,t), x_3=(2,0)$ in $\RR^2$. In this case,
$$ \sup_{i<j<k} \partial_1(x_i, x_j, x_k) \lesssim t^2 $$
but
$$ \sup_{i} \dist(x_i, L) \gtrsim t. $$
for any geodesic (line) $L$ in Euclidean space.

\end{remark}

We now focus on building up some preliminary facts needed for the proof of Lemma \ref{lem:ellinfty}. For a point $x$ in $\ell_\infty$, we write $x = (x^m)_{m=1}^\infty$.

\begin{lemma}\label{lem:nearlipgraph}
Let $S=(x_1, x_2, \dots, x_n)$ be an $n$-tuple of points in $\ell_\infty$. Assume that 
$$ |x_1^1 - x_n^1| = \|x_1 - x_n\|.$$

\begin{enumerate}[(a)]
\item If
$$\del_1(x_i, x_j, x_k) \leq h$$
 for each $i<j<k$, then
$$ \|x_i - x_j\| \leq |x_i^1 -x_j^1| + 2h$$
for each $i,j$. 

\item If $(x_1,\dots,x_n)$ is ordered and
$$ \|x_i - x_j\| = |x_i^1 -x_j^1|$$
 for each $i<j<k$, then
$$\del_1(x_i, x_j, x_k) =0 $$
 for each $i<j<k$.
\end{enumerate}
\end{lemma}
\begin{proof}
We begin with (a). To the contrary, suppose we had $i<j$ such that 
$$ \|x_i - x_j\| > |x_i^1 -x_j^1| + 2h.$$
Then
\begin{align*}
|x_1^1 - x_n^1| &= \|x_1 - x_n\|\\
&= \|x_1 - x_i\| + \|x_i - x_j\| + \|x_j - x_n\| - \del_1(x_1, x_i, x_n) - \del_1(x_i, x_j, x_n)\\
&\geq  \|x_1 - x_i\| + \|x_i - x_j\| + \|x_j - x_n\| - 2h\\
&> |x_1^1-x_i^1| + |x_i^1-x_j^1| + |x_j^1-x_n^1|\\
&\geq |x_1^1 - x_n^1|,
\end{align*}
which is a contradiction.

For (b), we use the order to write
$$ \del_1(x_i,x_j,x_k) = \|x_i-x_j\|+\|x_j-x_k\|-\|x_i-x_k\| = |x^1_i-x^1_j|+|x^1_j-x^1_k|-|x^1_i-x^1_k| = \del(x^1_i, x^1_j, x^1_k) = 0.$$

\end{proof}

\begin{lemma}\label{lem:extension}
Let $E\subseteq \RR$ be non-empty and let $f:E\rightarrow\RR$ satisfy
\begin{equation}\label{eq:1tlip}
|f(x) - f(y)| \leq |x-y| + t
\end{equation}
for some $t\geq 0$ and all $x,y\in E$. 

Then there is a $1$-Lipschitz $g:\RR\rightarrow\RR$ such that
$$|g(z)-f(z)| \leq t $$
for all $z\in E$.
\end{lemma}
\begin{proof}
Define
$$ g(x) = \inf\{\tilde{f}(x) : \tilde{f}\colon\RR\rightarrow \RR\text{ } 1\text{-Lipschitz}, \tilde{f}(z)\geq f(z) \text{ for all } z\in E\}.$$
Note that for any $x,z\in E$, we have
$$ \tilde{f}_z(x):= |x-z| + f(z) + t \geq f(x)$$
and so this $\tilde{f}_z$ is admissible in the above infimum.

This shows, first of all, that $g$ is finite on $E$ and hence a $1$-Lipschitz function satisfying $g\geq f$ on $E$. (See, e.g., \cite[Lemma 6.3]{He01}.) Moreover, for each $z\in E$, we have
$$ |g(z)-f(z)| = g(z) - f(z) \leq \tilde{f}_z(z) - f(z) = t,$$
which completes the proof.
\end{proof}

\begin{lemma}\label{lem:supmax}
Let $(x_1, x_2, \dots, x_n)$ be an ordered $n$-tuple of $r$-separated points in $\ell_\infty$. Assume that
$$\del_1(x_i, x_j, x_k) \leq h < \frac{r}{10}$$
 for each $i<j<k$.

Then there is an ordered $n$-tuple $(z_1, \dots, z_n)$ of $r/2$-separated points in $\ell_\infty$ and a coordinate $i_0\in\mathbb{N}$ such that
$$ \|z_i - x_i\| \leq h $$
for each $i$ and
$$ |z_1^{i_0} - z_n^{i_0}| = \|z_1 - z_n\|.$$
\end{lemma}
\begin{proof}
We may find an $i_0\in\mathbb{N}$ such that
$$|x_1^{i_0} - x_n^{i_0}| \geq \|x_1 - x_n\| - h.$$
We then increase or decrease $x_1^{i_0}$ by $h$ so that
\begin{equation}\label{eq:supmax}
|x_1^{i_0} - x_n^{i_0}| = \|x_1 - x_n\|,
\end{equation}
holds. We relabel the resulting point $z_1$. Let $z_i=x_i$ for $i\geq 2$. 

Of course, we have $\|z_i - x_i\| \leq h$ for each $i$, and so the $\{z_i\}$ have mutual distances at least $r-2h\geq r/2$. It remains only to show that $(z_1,\dots, z_n)$ is ordered.

If $i<j<k$, then
$$ \|z_i - z_k\| = \|z_i - z_j\| + \|z_j - z_k\| - \del_1(z_i,z_j, z_k) > \max(\|z_i-z_j\|,\|z_j-z_k\|),$$
using the fact that $\del_1(z_i,z_j,z_k) \leq  \del_1(x_i,x_j,x_k) + 6h \leq h+6h < r-2h.$

\end{proof}

Finally, we prove Lemma \ref{lem:ellinfty}.

\begin{proof}[Proof of Lemma \ref{lem:ellinfty}]
Let $(x_1, x_2, \dots, x_n)$ be an ordered $n$-tuple of $r$-separated points in $\ell_\infty$. Assume that
$$\del_1(x_i, x_j, x_k) \leq h < \frac{r}{1000} \text{ for all } i<j<k.$$
We would like to find a geodesic $L$ in $\ell_\infty$ such that
$$\dist(x_i, L) \leq 15h$$
for each $i\in\{1,\dots,n\}$.

We begin by applying Lemma \ref{lem:supmax} to find an ordered $n$-tuple $(z_1, \dots, z_n)$ of $r/2$-separated points in $\ell_\infty$ and a coordinate $i_0\in\mathbb{N}$ such that
$$ \|z_i - x_i\| \leq h $$
for each $i$ and
$$ |z_1^{i_0} - z_n^{i_0}| = \|z_1 - z_n\|.$$

Note that we have
$$ \del_1(z_i,z_j,z_k) \leq h+6h = 7h$$
for each $i<j<k$. We set $r'=r/2$ and $h'=7h$. Note that $h' < r'/20$.

For the remainder of the proof, we will assume without loss of generality that $i_0=1$, which we can achieve by reordering the coordinates.

Fix $m\geq 2$ for the moment. By Lemma \ref{lem:nearlipgraph}(a), we have that
$$ |z_i^m - z_j^m| \leq \|z_i - z_j\| \leq |z_i^1 -z_j^1| + 2h'$$
for all $i,j$. 

Let $E = \{z_1^1, \dots, z_n^1\}\subseteq \RR$ and define $f_m:E\rightarrow \RR$ by $f_m(z_i^1) = z_i^m$. The function $f_m$ then satisfies the ``coarse $1$-Lipschitz'' property \eqref{eq:1tlip}, with $t=2h'$. By Lemma \ref{lem:extension}, there is a $1$-Lipschitz $g_m:\RR\rightarrow \RR$ satisfying
\begin{equation}\label{eq:g-f}
|g_m(z_i^1) - z_i^m|= |g_m(z_i^1) - f_m(z_i^1)| \leq 2h'.
\end{equation}

We now use each $g_m$, for $m\geq 2$, to define points $y_1,\dots, y_n \in \ell_\infty$. For $i\in\{1,\dots,n\}$, let
$$ y_i = (z_i^1, g_2(z_i^1), g_3(z_i^1), \dots).$$
Because of \eqref{eq:g-f} and the definition of $f_m$, we have
$$ \|y_i - z_i\| \leq 2h'$$
for each $i$. 

It follows, first of all, that $(y_1, \dots, y_n)$ is ordered just as $(z_1, \dots, z_n)$ is. Indeed, if $i<j<k$, then
$$ \|y_i - y_k\| = \|y_i - y_j\| + \|y_j - y_k\| - \del_1(y_i,y_j, y_k) \geq \max(\|y_i-y_j\|,\|y_j-y_k\|),$$
using the fact that $\del_1(y_i,y_j,y_k) \leq  \del_1(z_i,z_j,z_k) + 6h' \leq h'+6h' < r'-4h'.$

Furthermore, since each $g_m$ is $1$-Lipschitz, we have that
$$ |y_i^1 - y_j^1| = |z_i^1 - z_j^1| \geq |g_m(z_i^1) - g_m(z_j^1)| = |y_i^m - y_j^m|$$
for each $i,j\in\{1,\dots,n\}$ and $m\geq 2$, and so
$$ \|y_i - y_j\| = |y_i^1 - y_j^1|$$
for each $i,j\in\{1,\dots,n\}$. 

Therefore, Lemma \ref{lem:nearlipgraph}(b) implies that
\begin{equation}\label{eq:del0}
 \del_1(y_i,y_j,y_k) = 0
\end{equation}
for all $i<j<k$. 

It follows that the union of line segments $[y_1, y_2] \cup [y_2, y_3] \cup \dots \cup [y_{n-1},y_n]$ is a geodesic segment passing within distance $2h'$ of each $z_i$. This segment can then be extended to a bi-infinite geodesic $L\subset \ell_\infty^d$ with the same property.

Finally, since $\|z_i - x_i\|\leq h$ for each $i$, we see that
$$ \dist(x_i,L) \leq h+2h' = 15h$$
for each $i$, which completes the proof.

\end{proof}

\section{$\beta$ numbers in uniformly convex Banach spaces}\label{sec:Banach}

In this section, we prove the following:

\begin{lemma}\label{lem:bxbi}
Let $(X,\|\cdot\|)$ be a Banach space with modulus of convexity $\delta$ satisfying \eqref{eq:powerconvex} with $c>0$ and $q\geq 2$. Let $\Gamma\subset X$ be a compact, connected subset. Then
$$ \bx^\Gamma(B) \lesssim \bi^\Gamma(B)^{2/q} $$
for each ball $B$ in $X$ with $\diam(B)\leq\frac{1}{10}\diam(\Gamma)$. 

The implied constant depends only on $c$ and $q$.
\end{lemma}

Using Theorem \ref{thm:upperbound}, we will then obtain Corollary \ref{cor:Banach}.

\begin{proof}[Proof of Corollary \ref{cor:Banach}]
Let $X$ be a Banach space satisfying the assumptions of the corollary and let $\Gamma$ be a doubling curve in $X$ with a multiresolution family $\cB$.

Exactly as in the proof of Corollary \ref{cor:ellinfty}, it suffices to show that
$$ \sum_{B\in \cB\setminus \cB_0} \bx^{\Gamma}(B)^{p}\diam(B) \lesssim \HH^1(\Gamma),$$
where
$$ \cB_0 = \{B\in\cB: \diam(B)\geq \frac{1}{10}\diam(\Gamma)\}.$$

Using Lemma \ref{lem:bxbi} and Theorem \ref{thm:upperbound}, we get
$$ \sum_{B\in \cB\setminus \cB_0} \bx^{\Gamma}(B)^{p}\diam(B) \lesssim \sum_{B\in \cB\setminus \cB_0} \bx^{\Gamma}(B)^{2p/q}\diam(B) \lesssim \HH^1(\Gamma),$$
which completes the proof.
\end{proof}

We now focus on proving Lemma \ref{lem:bxbi}.

\begin{lemma}\label{lem:lineconvexity}
Let $(X,\|\cdot\|)$ be a Banach space with modulus of convexity $\delta$ satisfying \eqref{eq:powerconvex} with $c>0$ and $q\geq 2$. Let $x,y,z$ be three points in $X$ with $\diam(\{x,y,z\})\leq r$, and let $L$ denote the line through $x$ and $z$. Then
$$ \left(\frac{\dist(y,L)}{r}\right)^q r \lesssim \del_1(x,y,z).$$
The implied constant depends only on $c$ and $q$.
\end{lemma}
\begin{proof}
Assume that $y\notin L$, otherwise the lemma is trivial.

Let $y_0$ denote a point on the segment $[x,z]\subseteq L$ such that
\begin{equation}\label{eq:y0}
\frac{\|y_0-x\|}{\|y_0-x\| + \|y_0-z\|} = \frac{\|y-x\|}{\|y-x\| + \|y-z\|}
\end{equation}
Note that the quantity on the right side of \eqref{eq:y0} is in  $[0,1]$, and the quantity on the left side ranges continuously from $0$ to $1$ as $y$ moves from $x$ to $z$ in the segment $[x,z]$, so such a $y_0$ exists. By simple algebra, this $y_0$ also satisfies
\begin{equation}\label{eq:y0'}
\frac{\|y_0-z\|}{\|y_0-x\| + \|y_0-z\|} = \frac{\|y-z\|}{\|y-x\| + \|y-z\|}.
\end{equation}

It follows that
\begin{equation}\label{eq:y0x}
 \|y_0-x\| = \|y-x\|\left(\frac{\|y_0-x\| + \|y_0-z\|}{\|y-x\| + \|y-z\|}\right) = \|y-x\|\left(\frac{\|x-z\|}{\|y-x\| + \|y-z\|}\right)  \leq \|y-x\|
\end{equation}
and
\begin{equation}\label{eq:y0z}
 \|y_0-z\| = \|y-z\|\left(\frac{\|y_0-x\| + \|y_0-z\|}{\|y-x\| + \|y-z\|}\right) = \|y-z\|\left(\frac{\|x-z\|}{\|y-x\| + \|y-z\|}\right)  \leq \|y-z\|.
\end{equation}

Let $y' = \frac{y+y_0}{2}$ and $h = \|y-y_0\| = 2\|y-y'\|$.

Equation \eqref{eq:y0x} implies that $y$ and $y_0$ are in the closed ball $B=\overline{B}(x,\|x-y\|)$. We now want to apply \eqref{eq:convexity} to these points in $B$, which we may rescale and translate to the unit ball. Doing so, we see that
$$\delta(h/\|x-y\|) \leq 1 - \frac{\|y'-x\|}{\|x-y\|}$$
or
\begin{equation}\label{eq:xy}
\|x-y\|\delta\left(\frac{h}{\|x-y\|}\right) \leq \|x-y\| - \|y'-x\|.
\end{equation}
Doing the same on the ball $\overline{B}(z,\|y-z\|)$ (which contains $y$ and $y_0$ by \eqref{eq:y0z}), we see that
\begin{equation}\label{eq:yz}
\|z-y\|\delta\left(\frac{h}{\|z-y\|}\right) \leq \|z-y\| - \|y'-z\|.
\end{equation}
Using equations \eqref{eq:xy} and \eqref{eq:yz}, we obtain that
\begin{align*}
\del_1(x,y,z) &= \|x-y\| + \|z-y\| - \|x-z\|\\
&\geq \|x-y\| + \|z-y\| - \|x-y'\| - \|z-y'\|\\
&\geq \|x-y\|\delta\left(\frac{h}{\|x-y\|}\right) + \|z-y\|\delta\left(\frac{h}{\|z-y\|}\right)\\
&\gtrsim h^q(\|x-y\|^{1-q} + \|z-y\|^{1-q})\\
&\gtrsim h^q r^{1-q}.
\end{align*}
Since $y_0\in L$, $h$ is an upper bound for $\dist(y,L)$. Therefore,
$$ \left(\frac{\dist(y,L)}{r}\right)^q r \leq \left(\frac{h}{r}\right)^q r \lesssim \del_1(x,y,z),$$
as the lemma states.
\end{proof}

\begin{proof}[Proof of Lemma \ref{lem:bxbi}]
Fix a compact, connected set $\Gamma\subseteq X$ and a ball $B=B(x,r)$ centered on $\Gamma$ with $\diam(B)\leq \frac{1}{10}\diam(\Gamma)$. Note that for such $B$, $r\approx\diam(B)$, so we may freely interchange these at the cost of some absolute constant factors.

Write $\bi$ and $\bx$ for $\bi^\Gamma$ and $\bx^\Gamma$, respectively. We may assume, by adjusting the implied constant, that $\bi(B)$ is small enough to apply Lemma \ref{lem:order}.

By that lemma, we obtain a $2\bi(B)\diam(B)$-net $N\subseteq  B$ with an order $o:N\rightarrow\RR$ satisfying
\begin{equation}\label{eq:ordernet}
o(x)<o(y)<o(z) \Rightarrow d(x,z)>\max\{d(x,y),d(y,z)\}
\end{equation}
for all  $x,y,z\in N$. Ordering $N$ according to $o$, we  write
$$ N = \{x_1, x_2, \dots, x_n\}.$$

Let $L$ be the line passing through $x_1$ and $x_n$. For each $x_i\in N$, we have
$$ \left(\frac{\dist(x_i,L)}{2r}\right)^q (2r) \lesssim \del_1(x_1, x_i, x_n) = \del(x_1,x_i, x_n) \leq \bi(B)^2,$$
by Lemma \ref{lem:lineconvexity}. Here the equality between $\del_1$ and $\del$ comes from the order property \eqref{eq:ordernet}.

Therefore, for each $x_i\in N$ we have
$$ \dist(x_i,L) \lesssim \bi(B)^{2/q} r.$$

It follows that for each $x\in B$, we have
$$ \dist(x,L) \lesssim \bi(B)^{2/q} r + \bi(B)r \lesssim \bi(B)^{2/q} r,$$
where the second inequality follows from the fact that $q\geq 2$.

It follows that
$$ \bx(B) \lesssim \bi(B)^{2/q},$$
as Lemma \ref{lem:bxbi} states.

\end{proof}

\section{A comparison with the Heisenberg group}\label{sec:Heisenberg}
Fix an isometric embedding $\iota\colon \bH\rightarrow \ell_\infty$. As different choice of embedding $\iota$ will not affect the results in this section, we suppress $\iota$ throughout and simply consider $\bH \subseteq \ell_\infty$,

Fix a curve $\Gamma\subseteq\bH\subseteq \ell_\infty$, with multiresolution family $\cB$ having inflation factor $A=10$. Recall that we defined a measure on $\cB$ by setting
$$ \cM_r(\cB') = \sum_{B\in\cB'} \beta^{K}_{\bH}(B)^r\diam(B) $$
for $\cB'\subseteq \cB$ and $r<4$. 

Theorem \ref{thm:LS} then implies that if $r<4$ and $\Gamma$ is a rectifiable curve in the Heisenberg group satisfying 
\begin{equation}\label{eq:longHcurve}
\HH^1(\Gamma)\geq 2C_r\diam(\Gamma),
\end{equation}
then 
\begin{equation}\label{eq:bHsum}
\cM_r(\cB) \geq \frac{1}{2C_r}\HH^1(\Gamma).
\end{equation}

On the other hand, by Corollary \ref{cor:ellinfty}, we have for each $p>1$, that
\begin{equation}\label{eq:bGHsum}
\sum_{B\in\cB} \beta^{\Gamma,\text{ net}}_{\ell_\infty}(B)^p \diam(B) \leq C_p\HH^1(\Gamma),
\end{equation}

Comparing these will yield Corollary \ref{cor:Hcurves}, which says that the set
$$ \cB_{c,q} = \{ B\in\cB: \beta^{\Gamma,\text{ net}}_{\ell_\infty}(B) > c\beta^{\Gamma}_{\bH}(B)^q \}$$
is small, measured by $\cM_r$.

\begin{proof}[Proof of Corollary \ref{cor:Hcurves}]
Let $p=r/q>1$.

Using \eqref{eq:bGHsum}, we have
\begin{align*}
\sum_{B\in \cB_{c,q}} \beta^{\Gamma}_{\bH}(B)^r\diam(B) &\leq c^{-p} \sum_{B\in \cB_{c,q}}\beta^{\Gamma,\text{ net}}_{\ell_\infty}(B)^p \diam(B)\\
&\leq c^{-p} C_p\HH^1(\Gamma)\\
&\leq 2c^{-p} C_p C_r \sum_{B\in\cB} \beta^{\Gamma}_{\bH}(B)^r\diam(B)
\end{align*}
If $c$ is chosen, depending on $q$, $r$, and $\delta$, so that
$$ 2c^{-p} C_p C_r < \delta,$$
then that completes the proof.
\end{proof}

In the next two remarks, we argue that Corollary \ref{cor:Hcurves} indicates some genuine difference between Euclidean space and the the Heisenberg group, in the sense that the analogous Euclidean statement is in some sense trivial while Corollary \ref{cor:Hcurves} is not.

\begin{remark}\label{rmk:Rncurves}
Fix an isometric embedding of $\RR^n$ into $\ell_\infty$ and a compact subset $K\subset \RR^n$. Given Theorem \ref{t:TST}, the natural analog of Corollary \ref{cor:Hcurves} might be to show, for a suitable choice of $c=c_{\delta}$, that
\begin{equation}\label{eq:Rncor}
\cM_2\left(\{B\in \cB: \beta^{\Gamma,\text{ net}}_{\ell_\infty}(B) > c\beta^{\Gamma}_{\RR^n}(B)^q \}\right) < \delta\HH^1(\Gamma).
\end{equation}
However, while true, this statement is trivial in $\RR^n$ for all $q\leq 2$, because in fact one may make the measured collection of balls not only small but \textit{empty} by an appropriate choice of $c$. (In particular, \eqref{eq:Rncor} holds for all sets in $\RR^n$, and not simply curves.) This is for the following reason:

One may show that for any ball $B$ in a multi-resolution family for $K$, we have
\begin{equation}\label{eq:epsilonbeta}
\beta^{\Gamma,\text{ net}}_{\ell_\infty}(B) \lesssim \beta^{K}_{\RR^n}(B)^2,
\end{equation}
with an absolute implied constant (once $A=10$ is fixed). Indeed, consider a $2^{-(n+1)}$-separated net $N$ in $K\cap B(x,A2^{-n})$, and let $\beta = \beta^{K}_{\RR^n}(B(A2^{-n}))$. As in the proof of Corollary \ref{cor:ellinfty}, we may assume $N$ is ordered and $\beta< A/100 = 1/10$. There is a line $L$ containing $N$ in its $2\beta$ neighborhood. Standard computations with the Pythagorean theorem show that the orthogonal projection $\pi$ onto $L$ satisfies
$$ ||\pi(y) - \pi(z)| - |y-z|| \lesssim \beta^2 \text{ for all } y,z\in N.$$
(Note that it is important here that $y,z\in N$ are well-separated.) It follows that for any ordered triple of points $a,b,c\in N$,
$$ \del_1(a,b,c) \lesssim \del_1(\pi(a),\pi(b),\pi(c)) + \beta^2 = \beta^2.$$
Hence, Lemma \ref{lem:ellinfty} implies that
$$\beta^{\Gamma,\text{ net}}_{\ell_\infty}(B) \lesssim \beta^2,$$
proving \eqref{eq:epsilonbeta}. 

In turn, \eqref{eq:epsilonbeta} shows that, if $q\leq 2$, the collection of balls summed over in \eqref{eq:Rncor} can be made empty by choosing $c$ sufficiently large, which justifies the statement made at the beginning of this remark.
\end{remark}

\begin{remark}\label{rmk:Hcurves}
In the Heisenberg group, on the other hand, Corollary \ref{cor:Hcurves} is non-trivial in the following sense: For each fixed $q\in (2, 4)$, one can construct a set $K\subseteq \bH$ and a ball $B$ centered on $K$ such that
$$ \frac{\beta^{K,\text{ net}}_{\ell_\infty}(B)}{\beta^{K, \text{ net}}_{\bH}(B)^{q}}$$ 
is arbitrarily large. Indeed, page 2 of \cite{LS15} gives an example, for any $\rho>0$, of a $1$-separated set $K=\{a,b,c\}$ in $B=B(a,2)\subseteq \mathbb{H}$ such that 
$$ \rho > \del(a,b,c) \gtrsim \beta^{K}_{\bH}(B)^{2},$$
with an absolute implied constant.

Since $\del(a,b,c)$ is easily seen to be a lower bound for $\beta^{K,\text{ net}}_{\ell_\infty}(B)$, we see that
$$ \frac{\beta^{K,\text{ net}}_{\ell_\infty}(B)}{\beta^{K}_{\bH}(B)^{q} } \gtrsim \frac{1}{\rho^{q-2}},$$
which can be made arbitrarily large.

In particular, given $q\in (2,4)$, there is no choice of $c=c_q$ that makes $\cB_{c,q}$ empty for all sets in $\bH$, as was the case for $\RR^n$ in the previous remark. 

Nonetheless, Corollary \ref{cor:Hcurves} shows that for \textit{curves} in the Heisenberg group that are a definite factor longer than their diameter, the collection $\cB_{c,q}$ must be small (for $q<4$ and appropriate $c=c_{q,r}$).
\end{remark}

\bibliography{TSTbib}{}
\bibliographystyle{plain}

\end{document}